\documentclass[12pt,reqno]{amsart}
\usepackage{latexsym, amsfonts,mathrsfs, amsmath, amssymb, amscd, epsfig}
\usepackage{mathrsfs}

\usepackage{amsthm}
\usepackage[all,cmtip,line]{xy}
\usepackage{array}
\usepackage{psfrag}
\usepackage{overpic}
\usepackage{bm}
\usepackage{setspace}

\newtheorem{theorem}{Theorem}
\newtheorem{lemma}{Lemma}[section]
\newtheorem{problem}{Problem}
\newtheorem{proposition}[lemma]{Proposition}
\newtheorem{definition}[lemma]{Definition}
\newtheorem{remark}[lemma]{Remark}

\theoremstyle{definition}

\newcommand \alp{\alpha}

\newcommand \vphi{\varphi}

\newcommand \Gam{\Gamma}
\newcommand \gam{\gamma}
\newcommand \om{\omega}
\newcommand \tx{\text}
\newcommand \R{\mathbb{R}}
\newcommand \til{\tilde}

\newcommand \der{\partial}
\newcommand \mcl{\mathcal}
\newcommand \ol{\overline}
\newcommand \Om{\Omega}
\newcommand \N{\mcl{N}}
\newcommand \q{{\bf q}}
\newcommand \s{{\bf s}}

\newcommand \corners{\Gamw}

\newcommand \Sen{S_{en}}

\newcommand \pex{p_{ex}}

\newcommand \rx{{\rm x}}
\newcommand \ry{{\rm y}}

\newcommand \rhos{\rho_c}

\newcommand \Gamen{\Gam_0}
\newcommand \Gamex{\Gam_L}
\newcommand \Gamw{\Gam_w}

\def\mcF{\mathcal{F}}
\def \msB {\mathscr{B}}
\def \msK {\mathscr{K}}

\def\mfF{\mathfrak{F}}
\def\mfG{\mathfrak{G}}

\def\mff{\mathfrak{f}}

\def\mfU{\mathfrak{U}}
\def\mfV{\mathfrak{V}}
\def\mfp{\mathfrak{p}}
\def\mfg{\mathfrak{g}}

\def\Div{{\rm div}}
\def\Curl{{\rm curl}}

\def\bu{\mathbf{u}}
\def\bv{\mathbf{v}}
\def\bI{\mathbf{I}}
\def\mE{\mathcal{E}}
\def\mfe{\mathfrak{e}}

\def\mclU{\mcl{U}}
\def\mclW{\mcl{W}}
\def\mclWen{\mcl{W}_{en}}

\newcommand \mfraka{\mathfrak{a}}
\newcommand \mfrakb{\mathfrak{b}}
\newcommand \mfrakc{\mathfrak{c}}
\newcommand \mfrakd{\mathfrak{d}}

\numberwithin{equation}{section}

\begin{document}
\title[Subsonic flow for full Euler-Poisson system]
{Subsonic solutions for steady Euler-Poisson system in two dimensional nozzles}

\author{Myoungjean Bae}
\address{M. Bae, Department of Mathematics\\
         POSTECH\\
         San 31, Hyojadong, Namgu, Pohang, Gyungbuk, Republic of Korea}
\email{mjbae@postech.ac.kr or mybjean@gmail.com}
\author{Ben Duan}
\address{B. Duan, Department of Mathematics\\
         POSTECH\\
         San 31, Hyojadong, Namgu, Pohang, Gyungbuk, Republic of Korea}
\email{bduan@postech.ac.kr}

\author{Chunjing Xie}
\address{C. Xie, Department of mathematics, Institute of Natural Sciences, Ministry of Education Key Laboratory of Scientific and Engineering Computing, Shanghai Jiao Tong University\\
800 Dongchuan Road, Shanghai, China
}
\email{cjxie@sjtu.edu.cn}

\date{}


\date{\today}

\maketitle


\begin{abstract}

In this paper, we prove the existence and stability of subsonic flows for steady full Euler-Poisson system in a two dimensional nozzle of finite length when imposing the electric potential difference on non-insulated boundary from a fixed point at the entrance, and prescribing the pressure at the exit of the nozzle.
The Euler-Poisson system for subsonic flow is a hyperbolic-elliptic coupled nonlinear system. One of the crucial ingredient of this work is the combination of Helmholtz decomposition for the velocity field and stream function formulation together. In terms of the Helmholtz decomposition, the Euler-Poisson system is rewritten as a second order nonlinear elliptic system of three equations and transport equations for entropy and pseudo-Bernoulli's invariant. The associated elliptic system in a Lipschitz domain with nonlinear boundary conditions is solved with the help of the estimates developed in \cite{BDX} based on its nice structure. The transport equations are resolved via the flow map induced by the stream function formulation. Furthermore, the delicate estimates for the flow map  give the uniqueness of the solutions.

\end{abstract}

\bigskip
\section{Introduction}
\label{section-1}

The following Euler-Poisson system
\begin{equation}\label{UnsteadyEP}
\left\{
\begin{aligned}
& \rho_t+ \Div (\rho \bu)=0, \\
& (\rho \bu)_t+\Div (\rho \bu \otimes \bu +p\bI)=\rho \nabla \Phi, \\
& (\rho \mE)_t +\Div(\rho\mE \bu +p\bu)=\rho \bu\cdot \nabla\Phi,\\
& \Delta \Phi=\rho-b(\rx)
\end{aligned}%
\right.
\end{equation}
models several physical flows including the propagation of electrons in
submicron semiconductor devices and plasmas (cf. \cite{MarkRSbook})(hydrodynamic
model), and the biological transport of ions for channel proteins (cf. \cite{Shu}). In the hydrodynamical model of semiconductor devices or plasmas, $\bu,
\rho$, $p$, and $\mE$ represent the macroscopic particle velocity, electron density,
pressure, and the total energy, respectively.  The electric potential $\Phi$ is  generated by the
Coulomb force of particles. $\bI$ is the identity matrix and $b(\rx)>0$ stands for the density of fixed,
positively charged background ions. The biological model describes the
transport of ions between the extracellular side and the cytoplasmic side of
the membranes(\cite{Shu}). In this case, $\rho$, $\rho \bu$, and $\Phi$ are the ion
concentration, the ions translational mass, and the electric potential,
respectively. The equations \eqref{UnsteadyEP} are closed with the aid of definition of total energy and the equation of state
\begin{equation}
\mE=\frac{|\bu|^2}{2}+\mfe\,\, \text{and}\,\,\quad p=p(\rho, \mfe),
\end{equation}
where $\mfe$ is the internal energy.
In this paper, we consider ideal polytropic gas for which the pressure $p$ is given by
\begin{equation}
\label{2d-1-b1}
p(\rho, \mfe)=(\gamma-1)\rho\mfe
\end{equation}
where $\gam>1$ is called the adiabatic constant. In terms of the entropy $S$, one also has
\begin{equation}\label{pentropy}
p(\rho, S)=\mfp \exp(\frac{S}{c_v})\rho^\gam,
\end{equation}
where $\mfp$ and $c_v$ are positive constants.
In fact, all results in this paper are true for flows with general equation of states.

The global weak solution and formation of singularity for classical solution for \eqref{UnsteadyEP} in one dimensional case was studied in \cite{ChenWang, WangChen}.
Recently, there are a few studies  for \eqref{UnsteadyEP} under the small perturbations of constant states, see \cite{Guo99, Guo, IP} and references therein. Our goal is to understand general steady solutions for \eqref{UnsteadyEP},  in particular, the steady transonic solutions for \eqref{UnsteadyEP}, and their stability.

In $\R^2$, let $u$ and $v$ denote the horizontal and vertical components of velocity ${\bf u}$. Then, the steady Euler-Poisson system is written as
\begin{equation}\label{2-a2}
\begin{cases}
(\rho u)_{x_1}+(\rho v)_{x_2}=0\\
(\rho u^2+p)_{x_1}+(\rho uv)_{x_2}=\rho \Phi_{x_1}\\
(\rho uv)_{x_1}+(\rho v^2+p)_{x_2}=\rho \Phi_{x_2}\\
(\rho u \msB)_{x_1}+(\rho v \msB)_{x_2}=\rho {\bf u}\cdot \nabla\Phi\\
\Delta\Phi=\rho-b(\rx)
\end{cases}
\end{equation}%
for $\rx=(x_1,x_2)\in\R^2$ where the Bernoulli's function $\msB$ is given by
\begin{equation}
\label{2d-1-a5}
\msB(\rho,{\bf u}, S)=\frac{|\bu|^2}{2}+e+\frac{p}{\rho}
=\frac{|\bu|^2}{2}+\frac{\gam\mfp}{\gam-1} \exp(\frac{S}{c_v})\rho^{\gam-1}.
\end{equation}

One of the main difficulties is that the equations \eqref{2-a2} change type when the flow speed varies from subsonic ($|\bu|< \sqrt{p_\rho(\rho, S)}$) to supersonic ($|\bu|> \sqrt{p_\rho(\rho, S)}$).
There have been a few studies on transonic solutions of the Euler-Poisson system(cf.\cite{MarkPhase, Gamba1d, LRXX, LuoXin, RosiniPhase}). A special transonic solution for one dimensional Euler-Poisson equations was investigated in \cite{MarkPhase}. General transonic solutions for one dimensional Euler-Poisson equations was established in \cite{Gamba1d} by vanishing viscosity method. However, the fine structure of the solutions obtained in \cite{Gamba1d} other than BV functions is not clear. A thorough study  of one-dimensional transonic shock solutions to the Euler-Poisson equations with a constant background charge  was given in \cite{LuoXin}. One of interesting phenomenon in \cite{LuoXin} is that there are multiple solutions for the one dimensional system under given boundary conditions.  So it is natural to ask whether these one dimensional transonic solutions are dynamically and structurally stable.
The dynamical stability of transonic shock solutions for one dimensional Euler-Poisson system is achieved in \cite{LRXX}. Furthermore, it reveals that the effect of electric force in the Euler-Poisson system has the same physical effect as the geometry of divergent domain in the Euler system, see \cite{LuoXin, LRXX}. This makes multidimensional structural stability problems for transonic shocks of the Euler-Poisson equations physically important and interesting.  Recently, there are extensive studies and significant progress on transonic shock solutions of the Euler system, see \cite{Bae-F, CF1, Schen-Yu1, LXY1, XinYinCPAM} and references therein. However, the Poisson equation makes the fluid variables in the Euler-Poisson system coupled in a nonlocal and nonlinear way so that multidimensional stability problems for transonic flows of the Euler-Poisson system are essentially open.  Therefore, the Euler-Poisson system is worthy of studying not only physically but also mathematically.  An effort in studying multidimensional transonic flows for the Euler-Poisson equations was made in \cite{GambaMorawetz} for a viscous approximation of transonic solutions in two dimensional case. However, the zero viscosity limit in \cite{GambaMorawetz} remains an open problem.

As the first step to investigate stability of multidimensional transonic flow of the Euler-Poisson system under perturbations of exit pressure, we establish the unique existence and stability of subsonic flows of steady Euler-Poisson system under perturbations of the exit pressure and electric potential difference on non-insulated boundary.

When the current flux is sufficiently small, the existence of subsonic solutions was proved in \cite{DeMark1d, DeMark3d, Guo, MarkZAMP,Yeh, Hattori}. The existence of solutions with large data is also obtained in \cite{DeMark1d, DeMark3d} when the convection term has a sufficiently small parameter.  To our knowledge, the first result about structural stability of subsonic solutions with large variations for the Euler-Poisson system under perturbations of exit pressure is given in \cite{BDX}. When the flows is subsonic, the Euler-Poisson system for isentropic potential flow can be written as a second order quasi-linear elliptic system for the velocity potential $\vphi$ and the electric potential $\Phi$. In \cite{BDX}, it is discovered that a linearized system of the elliptic system for $(\vphi,\Phi)$ has a nice structure which enables a priori $H^1$ estimate for weak solutions of the associated linear elliptic system, which is one of the key ingredients to establish unique existence and stability of subsonic potential flow in multidimensional nozzle under perturbations of exit pressure. Similar result is obtained in \cite{BDX2} for two dimensional Euler-Poisson system with gravitational potential using stream function formulation.

In this paper, we study the existence and stability of two dimensional subsonic flows for full Euler-Poisson system with non-zero vorticity and non-constant entropy. The study for the flows with non-constant entropy and non-zero vorticity is important because the transonic shock problem for isentropic model might be ill-posed as the case in gas dynamics \cite{XinYinCPAM}. Our goal is to  prove unique existence and stability of subsonic solutions to \eqref{2-a2} in two dimensional nozzle of finite length under perturbations of exit pressure and electric potential difference on non-insulated boundary from a fixed point at the entrance. The direct computation shows that it is hard to get the existence of solutions for the system for the stream function and the electric potential.  To overcome this difficulty, we  write ${\bf u}=\nabla\vphi+\nabla^{\perp}{\psi}$ for $\nabla^{\perp}=(\der_{x_2}, -\der_{x_1})$.  There are two advantages of this decomposition. The first is that \eqref{2-a2} can be written as a second order quasi-linear system for $(\vphi,\psi, \Phi)$ and two homogeneous transport equations for the entropy and pseudo-Bernoulli's invariant. The second is that the system for $(\vphi,\psi,\Phi)$ can be weakly decomposed into a second order elliptic system for $(\vphi,\Phi)$ and a Poisson equation for $\psi$, and the elliptic system for $(\vphi,\Phi)$ has the nice structure so that a priori $H^1$ estimate of weak solutions is possible. Because the transport equations do not gain regularity, we use the stream function formulation to represent the solution in terms of the flow map which can be also used to prove the uniqueness of solutions.

As far as we know, this is the first work which treats Euler-Poisson system using both Helmholtz decomposition and stream function approach together. This combination of two different methods requires subtle and careful analysis to solve  \eqref{2-a2} with nonlinear boundary condition.

There are a lot of other studies on the Euler-Poisson system with relaxations, see \cite{LuoNX, LiMark, Huang, Guo} and the references therein.

The rest of the paper is organized as follows. The problem and main result are stated in Section \ref{section-main-thm}. In Section \ref{sectionreformulation}, we reformulate \eqref{2-a2} as a nonlinear problem  by Helmholtz decomposition, and state unique solvability of the reformulated problem as Theorem \ref{nf-theorem3}. Then we prove Theorem \ref{nf-theorem1} by  Theorem \ref{nf-theorem3}. In Section \ref{section-nlbvp}, unique solvability of the elliptic system for $(\vphi,\psi,\Phi)$ is proved. Finally, Theorem \ref{nf-theorem3} is proved in Section \ref{section-pf-theorem2} with the careful analysis for the flow map.

\section{The Problem and Main Results}
\label{section-main-thm}

Let  $b_0>0$ be a fixed constant.
Consider a solution $(\rho, u,v, p, \Phi)$ of \eqref{2-a2} with $b=b_0$, $v=0= \Phi_{x_2}$, $\rho>0$ and $u>0$. Set $E:=\Phi_{x_1}$. Then $(\rho, u, p, E)$ satisfy the nonlinear ODE system
\begin{equation}
\label{2d-1-a1}
\begin{cases}
(\rho u)'=0\\
S'=0\\
\msB'=E\\
E'=\rho-b_0
\end{cases}
\end{equation}
where $'$ denotes the derivative with respect to $x_1$. It follows from the first two equations in \eqref{2d-1-a1} that one has
\begin{equation}
\label{2d-1-a2}
\rho u= {J_0}\quad \text{and} \quad S= S_0
\end{equation}
where the constants $J_0>0$ and $S_0>0$ are determined by initial data. Therefore, \eqref{2d-1-a1} can be written as an ODE system for $(\rho, E)$ as follows:
\begin{equation}
\label{2d-1-a3}
\begin{cases}
\rho'=\frac{\rho E}{\gam \mfp \exp(\frac{S_0}{c_v})\rho^{\gam-1}-\frac{J_0^2}{\rho^2}},\\
E'=\rho-b_0.
\end{cases}
\end{equation}

\begin{proposition}
\label{proposition-21}
Fix $b_0>0$. Given constants $J_0>0$, $S_0>0$, $\rho_0>\rhos:= \left(\frac{J_0^2}{\gam \mfp \exp(\frac{S_0}{c_v})}\right)^{\frac{1}{\gam+1}}$ and $E_0$, there exists positive constants $L$, $\rho_\sharp$, ${\rho}^\sharp$, and $\nu_0$ such that the initial value problem \eqref{2d-1-a3} with
\begin{equation}
\label{iv}
\rho(0)=\rho_0\quad \text{and}\quad
E(0)=E_0
\end{equation}
has a unique smooth solution $(\rho(x), E(x))$ on the interval $[0,L]$ satisfying
\begin{equation}
\label{2-a7}
0< {\rho}_\sharp \le \rho\le {\rho}^\sharp \quad\tx{and}\quad \gam \mfp \exp(\frac{S_0}{c_v})\rho^{\gam-1}-\frac{J_0^2}{\rho^2}\ge \nu_0\;\;\tx{on}\;\;[0,L].
\end{equation}
\end{proposition}
The detailed proof for Proposition \ref{proposition-21} can be found in \cite{LuoXin}, so we omit it.

From now on, suppose that $J_0>0$, $S_0>0$, $b_0>0$, $\rho_0>\rhos$ and $E_0$ are fixed constants. Let $L>0$ be as in Proposition \ref{proposition-21}.
Set $\N:=(0, L)\times (0,1)$ be a fixed flat nozzle in $\R^2$.
The nozzle boundary $\partial\N$ consists of
$$
\Gamen=\{0\}\times[0,1],\quad \Gamex=\{L\}\times[0,1], \quad\text{and}\,\,\Gamma_w=(0, L)\times \{0, 1\}$$
which are the  entrance, the exit, and the insulated boundary of $\N$, respectively.  For the unique solution $(\rho, E)$ of \eqref{iv} obtained in Proposition \ref{proposition-21}, let $u$ and $p$ be defined by \eqref{2d-1-a2} and \eqref{pentropy}, respectively.
From now on, denote $\bar{\rho}(\rx)=\rho(x_1)$, $\bar{\bf u}(\rx)=(u(x_1),0)$, $\bar p(\rx)=p(x_1)$, respectively, for $\rx=(x_1,x_2)\in\N$. Later on, we also denote $\bar u(x_1)=u(x_1)$. Define  $\Phi_0(\rx)$ and $\vphi_0(\rx)$ by
\begin{equation}
\label{2-b6}
\Phi_0(\rx)=\int_{0}^{x_1}E(t)dt\quad\text{and}\quad
\vphi_0(\rx)=\int_0^{x_1}\bar u(t)\;dt
\end{equation}
for $\rx=(x_1,x_2)\in\N$.
\begin{remark}
It follows from \eqref{2-b6} that $\Phi_0$ can also be regarded as the electric potential difference between the points in $\N$ and a given point at the entrance.
\end{remark}

It is easy to see that $(\rho, {\bf u}, p, \Phi)=(\bar{\rho},\bar{\bf u}, \bar p,\Phi_0)$ satisfy \eqref{2-a2} in $\N$.

\begin{definition}
\label{definition-1}
The solution $(\bar{\rho},\bar{\bf u}, \bar p, \Phi_0)$ of \eqref{2-a2} that is obtained in Proposition \ref{proposition-21} and satisfies \eqref{2-a7} in $\N$ is called a \emph{subsonic background solution} associated with the parameters  $b_0$, $S_0$, $J_0$, $\rho_0$, $E_0$ and $L$.
\end{definition}
Before we state our main problem and main result, weighted H\"{o}lder norms are introduced first. For a bounded connected open set $\Om\subset\R^n$, let $\Gam$ be a closed portion of $\der\Om$. For $\rx, \ry\in\Om$, set
\begin{equation*}
\delta_{\rx}:={\rm dist}(\rx,\Gam)\quad \text{and}\quad  \delta_{\rx,\ry}:=\min(\delta_{\rx},\delta_{\ry}).
\end{equation*}
Given $k\in\R$, $\alp\in(0,1)$ and $m\in \mathbb{Z}^+$, define the standard H\"{o}lder norms by
\begin{align*}
&\|u\|_{m,\Om}:=\sum_{0\le|\beta|\le m}\sup_{\rx\in \Om}|D^{\beta}u(\rx)|,\quad
[u]_{m,\alp,\Om}:=\sum_{|\beta|=m}\sup_{\rx, \ry\in\Om,\rx\neq  \ry}\frac{|D^{\beta}u(\rx)-D^{\beta}u(\ry)|}{|\rx-\ry|^{\alp}},
\end{align*}
and the weighted H\"{o}lder norms by
\begin{align*}
&\|u\|_{m,0,\Om}^{(k,\Gam)}:=\sum_{0\le|\beta|\le m}\sup_{\rx\in \Om}\delta_{\rx}^{\max(|\beta|+k,0)}|D^{\beta}u(\rx)|,\\
&[u]_{m,\alp,\Om}^{(k,\Gam)}:=\sum_{|\beta|=m}\sup_{\rx,\ry\in\Om, \rx\neq \ry}\delta_{\rx,\ry}^{\max(m+\alp+k,0)}\frac{|D^{\beta}u(\rx)-D^{\beta}u(\ry)|}{|\rx-\ry|^{\alp}},\\
&\|u\|_{m,\alp,\Om}:=\|u\|_{m,\Om}+[u]_{m,\alp,\Om},\quad
\|u\|_{m,\alp,\Om}^{(k,\Gam)}:=\|u\|_{m,0,\Om}^{(k,\Gam)}+[u]_{m,\alp,\Om}^{(k,\Gam)},
\end{align*}
where $D^{\beta}$ denotes $\der_{x_1}^{\beta_1}\cdots\der_{x_n}^{\beta_n}$ for a multi-index $\beta=(\beta_1,\cdots,\beta_n)$ with $\beta_j\in\mathbb{Z}_+$ and $|\beta|=\sum_{j=1}^n\beta_j$. $C^{m,\alp}_{(k,\Gam)}(\Om)$ denotes the completion of the set of all smooth functions whose $\|\cdot\|_{m,\alp,\Om}^{(k,\Gam)}$ norms are finite. For simplicity,  denote $\|\cdot\|_{0,\alp,\Om}$ by $\|\cdot\|_{\alp,\Om}$. For a vector function $\bv=(v_1, \cdots, v_d)$, denote $\|\bv\|_{m, \alpha, \Omega}^{(k, \Gamma)}=\sum_{i=1}^d \|v_i\|_{m, \alpha, \Omega}^{(k, \Gamma)}$.

Our main concern is to solve the following problem.
\begin{problem}
\label{problem1}
Fix  a point $\bar{\rx}=(\bar x_1,0)\in\Gamen$. Suppose that $(b,\Sen, \msB_{en}, \Phi_{bd},\pex)$ satisfy
\begin{equation*}
\begin{split}
&\|b-b_0\|_{\alp,\N}+\|(\Sen,\msB_{en})-(S_0,\msB_0)\|_{1,\alp,\Gamen}\\
 &+ \|\Phi_{bd}-\Phi_0\|_{2,\alp,\Gamen\cup \Gamex}^{(-1-\alp,\partial(\Gamen\cup \Gamex ))}
+\|\pex-\bar{p}\|_{1,\alp,\Gamex}^{(-\alp,\der\Gamex)}\le \sigma
\end{split}
\end{equation*}
with sufficiently small $\sigma>0$ where $\msB_0=\msB(\bar\rho(\bar \rx), \bar\bu(\bar \rx), S_0)$ and $\alpha\in (0, 1)$,
find a solution $(\rho, \bu, p, \Phi)$ to \eqref{2-a2} in $\N$ satisfying the boundary conditions
\begin{align}
\label{2-b9}
&S=S_{en},\quad \msB=\msB_{en},\;\; {\bf u}\cdot{\bm \tau}=0\quad\tx{on}\;\;\Gam_0\\
\label{2-b9-1}
&\Phi-\Phi(\bar{\rx})=\Phi_{bd}\;\;\tx{on}\;\;\Gam_0\cup\Gamex\\
\label{2-b9-4}
&\bu \cdot{\bf n}_w=\der_{{\bf n}_w}\Phi=0\;\;\tx{on}\;\;\Gam_w,\\
\label{2-b9-3}
&p=p_{ex}\;\;\tx{on}\;\;\Gamex.
\end{align}
\end{problem}
\begin{remark}
It is easy to see that the background solutions $(\bar\rho, \bar\bu, \bar p, \Phi_0)$ solves the Problem \ref{problem1} with $\sigma=0$. So Problem \ref{problem1} can be regarded as the stability problem for the background solutions.
\end{remark}
\begin{remark}
It follows from the similar analysis as that in \cite{BDX} that the boundary conditions \eqref{2-b9}-\eqref{2-b9-3} in the one dimensional setting are equivalent to prescribe the density at both the entrance and the exit for the system \eqref{2d-1-a3}.
\end{remark}

Before we state our main result in this paper, let us introduce the following notation: we say a constant $C$ depending on the data if it depends on $b_0$, $S_0$, $J_0$, $\rho_0$, $E_0$ and $L$.

Our main result in this paper is as follows.

\begin{theorem}
\label{nf-theorem1}
Suppose that $(\bar{\rho},\bar{\bf u}, \bar p, \Phi_0)$  is  the subsonic background solution in $\N$ associated with parameters $b_0>0$, $S_0>0$, $J_0>0$, $\rho_0>\rhos$, $E_0$, and $L$.
Assume that $\Phi_{bd}$ satisfies the compatibility condition
\begin{equation}
\label{compt}
\der_{x_2}\Phi_{bd}=0\quad\tx{on}\quad \ol{\Gam_w}\cap(\ol{\Gamen}\cup\ol{\Gamex}).
\end{equation}
\begin{itemize}
\item[(a)](Existence)  There exists a $\sigma_1>0$ depending on the data and $\alp$ so that if
\begin{equation} \label{theorem1-est1}
\om_1(b)+\om_2(\Sen, \msB_{en})+\om_3(\Phi_{bd},\pex)\le \sigma_1,
\end{equation}
where
\begin{equation}\label{theorem1-as1}
\begin{aligned}
&\om_1(b):=\|b-b_0\|_{\alp,\N},\quad \om_2(\Sen, \msB_{en}):=\|(\Sen,\msB_{en})-(S_0,\msB_0)\|_{1,\alp,\Gamen},\\
&\om_3(\Phi_{bd},\pex):=\|\Phi_{bd}-\Phi_0\|_{2,\alp,\Gamen\cup\Gamex}^{(-1-\alp,\partial(\Gamen\cup\Gamex))}
+\|\pex-\bar{p}\|_{1,\alp,\Gamex}^{(-\alp,\der\Gamex)},
\end{aligned}
\end{equation}
then the boundary value problem \eqref{2-a2} with \eqref{2-b9}--\eqref{2-b9-3} has a solution $(\rho, \bu, p,\Phi)$ satisfying
\begin{equation}
\label{theorem1-est2}
\begin{split}
&\|(\rho, {\bf u}, p)-(\bar{\rho},\bar{\bf u}, \bar p)\|_{1,\alp,\N}^{(-\alp,\corners)}
+\|\Phi-\Phi_0\|_{2,\alp,\N}^{(-1-\alp,\corners)}\\
&\le C\left(\om_1(b)+\om_2(\Sen, \msB_{en})+\om_3(\Phi_{bd},\pex)\right)
\end{split}
\end{equation}
where the constant $C$ depends only on the data and $\alp$.

\item[(b)](Uniqueness)  There exists a $\sigma_2>0$ depending on the data, $\alpha$, and $\mu$  such that if
    \begin{equation}
       \label{theorem1-est3}
       \om_1(b)+\om_2(\Sen, \msB_{en})+\om_3(\Phi_{bd},\pex)+\om_4(\Sen, \msB_{en}, \Phi_{bd})\le \sigma_2,
    \end{equation}
with $\alp \in (\frac{1}{2}, 1)$ and $\mu \in (2, \infty)$    where
    \begin{equation}
    \om_4(\Sen, \msB_{en},\Phi_{bd}):=\|(\Sen,\msB_{en}-\Phi_{bd})-(S_0,\msB_0)\|_{W^{2,\mu}(\Gamen)},
    \end{equation}
then the solution $(\rho, \bu, p,\Phi)$ obtained in (a) is unique.
\end{itemize}
\end{theorem}
\begin{remark}
As same as that in \cite{BDX}, we can also prove the stability of subsonic flows under small perturbations of the nozzle boundary.
\end{remark}

\section{Reformulation of the Problem}
\label{sectionreformulation}
\subsection{Reformulation for the Euler-Poisson equations and proof of Theorem \ref{nf-theorem1}}
Suppose that $\rho>0$ and $u>0$.
It is easy to see that
 $(\rho, \bu, p, \Phi)\in (C^1(\N))^4\times C^2(\N)$ solve \eqref{2-a2} if and only if they satisfy
\begin{align}
\label{2d-2-a1}
&(\rho u)_{x_1}+(\rho v)_{x_2}=0,
\end{align}
\begin{align}
\label{2d-2-a2}
&(\rho uv)_{x_1}+(\rho v^2+p)_{x_2}=\rho\Phi_{x_2},\\
\label{2d-2-a3}
&\rho{\bf u}\cdot\nabla S=0,\\
\label{2d-2-a4}
&\rho{\bf u}\cdot\nabla\msK=0,\\
\label{2d-2-a5}
&\Delta\Phi=\rho-b(\rx),
\end{align}
where $\msK$ is defined by
\begin{equation}
\label{2d-2-a7}
\msK=\msB-\Phi.
\end{equation}
 We call $\msK$ \emph{pseudo-Bernoulli's invariant} in the sense that $\msK$ is a constant along each integral curve of ${\bf u}$.

We apply the Helmholtz decomposition to rewrite the velocity field ${\bf u}$ as a summation of a gradient and a divergence free field.
Given velocity field $\bu=(u, v)$, denote $\Curl \bu = \partial_{x_2}u-\partial_{x_1}v$. Then one can find a function $\psi$ satisfying
\begin{equation}
\label{ompsi}
\begin{cases}
\Delta \psi = \Curl\;{\bf u}\;\;\tx{in}\quad\N\\
\der_{x_1}\psi=0\;\;\tx{on}\;\;\Gamen\cup \Gamex,\quad \psi=0\;\; \text{on}\;\; \Gamw.
\end{cases}
\end{equation}
Then, it is easy to see that
$$
\Curl (\bu -\nabla^\perp \psi)=0
$$
where $\nabla^\perp\psi =(\partial_{x_2}\psi, -\partial_{x_1}\psi)$. Therefore, there exists a function $\varphi$ satisfying
$\nabla \varphi=\bu -\nabla^\perp \psi$.
Hence the velocity field $\bu$ can be expressed as
\begin{equation}
\label{2d-2-a6}
\bu=\nabla\vphi+\nabla^{\perp}\psi.
\end{equation}
One may say that the velocity field $\bu$ has a decomposition of compressible irrotational part $\nabla \varphi$ and incompressible vortical part $\nabla^\perp \psi$.

It follows from \eqref{2d-1-a5}, \eqref{2d-2-a7}, and \eqref{2d-2-a6} that the density $\rho$ can be written as
\begin{equation}
\label{rhotil}
\rho=H(S, \msK+\Phi-\frac 12|\nabla\vphi+\nabla^{\perp}\psi|^2)
\end{equation}
where the function $H(S,\zeta)$ is defined by
\begin{equation}
\label{2d-2-a8}
H(S,\zeta)= \left[\frac{(\gam-1)\zeta}{\gam\mfp \exp(\frac{S}{c_v})}\right]^{\frac{1}{\gam-1}}.
\end{equation}
The straightforward computations for \eqref{2d-2-a1} and \eqref{2d-2-a2} give
\begin{equation}
\label{nf-2-b1}
u(v_{x_1}-u_{x_2})=T(\rho,S) S_{x_2} -\msK_{x_2}
\end{equation}
where $T$ is the temperature defined by
\begin{equation*}\label{2d-2-a9}
T(\rho, S)= \frac{c_v \mfp}{\gam-1} \exp(\frac{S}{c_v}) \rho^{\gam-1}.
\end{equation*}

The system \eqref{2d-2-a1}--\eqref{2d-2-a5} can be written as the following nonlinear system for $(\vphi,\psi, \Phi, S, \msK)$:
\begin{align}
\label{nf-2-b6}
&\Div\bigl(H(S, \msK+\Phi-\frac{1}{2}|\nabla\vphi+\nabla^{\perp}\psi|^2)
(\nabla\vphi+\nabla^{\perp}\psi)\bigr)=0,\\
\label{nf-2-b9}
&\Delta \Phi=H(S, \msK+\Phi-\frac{1}{2}|\nabla\vphi+\nabla^{\perp}\psi|^2) -b,\\
\label{nf-2-b7}
&\Delta \psi=-\frac{T(\rho, S)\der_{x_2} S-\der_{x_2}\msK}{\partial_{x_1}\vphi+\partial_{x_2}\psi},\\
\label{eqSK}
&\rho(\nabla\vphi+\nabla^{\perp}\psi)\cdot \nabla S=
\rho(\nabla\vphi+\nabla^{\perp}\psi)\cdot \nabla \msK=0,
\end{align}
where $\rho=H(S, \msK+\Phi-\frac{1}{2}|\nabla\vphi+\nabla^{\perp}\psi|^2)$.


If ${\bf u}$ is given by \eqref{2d-2-a6} with $\psi$ satisfying \eqref{ompsi}, then ${\bf u}$ satisfy \eqref{2-b9} and \eqref{2-b9-4} if $\vphi$ satisfies
\begin{equation}\label{nf-3-a5}
\vphi=0\;\;\tx{on}\;\;\Gam_0\quad\tx{and}\quad \der_{{\bf n}_w}\vphi=0 \;\;\tx{on}\;\; \Gam_w.
\end{equation}
It follows from \eqref{pentropy} and \eqref{rhotil} that the boundary condition \eqref{2-b9-3} can be expressed by$(\vphi,\psi,\Phi, S,\msK)$ as follows:
\begin{equation}
\label{bc-vphi-ex}
\mfp \exp(\frac{S}{c_v}) H^\gam(S, \msK+\Phi -\frac{1}{2} |\nabla\vphi+\nabla^{\perp}\psi|^2)=\pex
\;\;\tx{on}\;\;\Gamex.
\end{equation}

Note that in Problem \ref{problem1}, the electric potential $\Phi$ at $\bar \rx$ has a freedom. Without loss of generality, we choose $\Phi(\bar\rx)=0$. From now on, denote $\mfU = (\vphi,\psi,\Phi,S,\msK)$, $\msK_{en}=\msB_{en}-\Phi_{bd}$, and $\mfU_0 = (\vphi_0, 0 ,\Phi_0,S_0,\msK_0)$ with $\msK_0=\msB_0$.

We have seen that if $\mfU \in (C^2(\N))^3\times (C^1(\N))^2$ satisfy \eqref{nf-2-b6}--\eqref{eqSK} and the boundary conditions \eqref{nf-3-a5}, \eqref{bc-vphi-ex}, and
\begin{align}
\label{s-bc-Phi}
&\Phi=\Phi_{bd}\,\, \tx{on}\;\;\Gamen\cup\Gamex,\quad  \der_{{\bf n}_w}\Phi=0\,\,\tx{on}\;\;\Gam_w,\\
\label{s-bc-psi}
& \partial_{x_1}\psi=0 \,\,\text{on}\,\,\Gamen\cup \Gamex\,\,\quad \text{and}\quad \,\, \psi=0\,\,\text{on}\,\,\Gamw,\\
\label{bcSK}
&S=S_{en}\,\,\text{and}\,\, \msK=\msK_{en} \,\, \text{on}\,\, \Gam_0,
\end{align}
then $(\rho, {\bf u}, p,\Phi)$ given by \eqref{rhotil}, \eqref{2d-2-a6} and \eqref{pentropy} solve \eqref{2-a2} in $\N$ with \eqref{2-b9}--\eqref{2-b9-3} provided that $\rho>0$ and $u>0$. Note that $\mfU_0$ solves \eqref{nf-2-b6}--\eqref{eqSK} with \eqref{nf-3-a5}--\eqref{bcSK} with $(\pex,\Phi_{bd},\Sen, \msK_{en})=(\bar p(L,x_2),$ $\Phi_0(\rx),S_0,\msK_0)$.

The straightforward computations show that the equation \eqref{nf-2-b6} is an elliptic equation for subsonic flows when $|\nabla\varphi+\nabla^\perp\psi|^2<c^2(\rho, S)$ where $c(\rho, S)=\sqrt{p_\rho(\rho, S)}$ is the local sound speed, and a hyperbolic equation for supersonic flows.

We have the following theorem for the problem \eqref{nf-2-b6}--\eqref{eqSK} with boundary conditions \eqref{nf-3-a5}--\eqref{bcSK}.
\begin{theorem}
\label{nf-theorem3}
 Suppose that  $(\bar{\rho}, \bar{\bf u}, \bar p, \Phi_0)$ is the subsonic background solution associated with the parameters  $b_0>0$, $S_0>0$, $J_0>0$, $\rho_0>\rhos$, $E_0$ and $L$.
Assume that $\Phi_{bd}$ satisfies the compatibility condition \eqref{compt}.
\begin{itemize}
\item[(a)](Existence) There exists a $\sigma_3>0$ depending on data  and $\alp$ so that if
    \begin{equation}
   \label{theorem3-est1}
   \om_1(b)+\om_2(\Sen, \msB_{en})+\om_3(\Phi_{bd},\pex)\le \sigma_3,
    \end{equation}
then the boundary value problem \eqref{nf-2-b6}--\eqref{eqSK} with \eqref{nf-3-a5}--\eqref{bcSK} has a solution $\mfU$ satisfying
\begin{equation}
\label{theorem3-est2}
\begin{split}
&\|(\vphi,\psi,\Phi)-(\vphi_0,0,\Phi_0)\|_{2,\alp,\N}^{(-1-\alp,\corners)}
+\|(S,\msK)-(S_0,\msK_0)\|_{1,\alp,\N}\\
&\le C\left(\om_1(b)+\om_2(\Sen, \msB_{en})+\om_3(\Phi_{bd},\pex)\right)
\end{split}
\end{equation}
where the constant $C$ depends only on the data and $\alp$.

\item[(b)](Uniqueness) Fix $\alp\in(\frac 12, 1)$ and  $\mu \in(2,\infty)$. Then there exists a $\sigma_4>0$ depending on data, $\alp$, and $\mu$ so that if
    \begin{equation}
       \label{theorem3-est3}
      \om_1(b)+\om_2(\Sen, \msB_{en})+\om_3(\Phi_{bd},\pex)+\om_4(\Sen, \msB_{en}, \Phi_{bd})\le \sigma_4,
    \end{equation}
then the solution $\mfU$  in (a) is unique.
\end{itemize}
\end{theorem}
In the following we show that Theorem \ref{nf-theorem1} is a consequence of Theorem \ref{nf-theorem3}.
\begin{proof}
[Proof of Theorem \ref{nf-theorem1}.]
Given  $(b,\Sen,\msB_{en},\Phi_{bd}, \pex)$ satisfying \eqref{theorem3-est1}, let $\mfU$ be a solution to the boundary value problem \eqref{nf-2-b6}--\eqref{eqSK} with \eqref{nf-3-a5}--\eqref{bcSK}. Let ${\bf u}$,  $\rho$,  and $p$ be given by \eqref{2d-2-a6}, \eqref{rhotil}, and \eqref{pentropy}, respectively.
Then \eqref{theorem3-est2} yields
\begin{equation}
\label{2d-2-b1}
\|u-\bar u\|_{0,\N}+\|\rho-\bar{\rho}\|_{0,\N}\le C( \om_1(b)+\om_2(\Sen, \msB_{en})+\om_3(\Phi_{bd},\pex))
\end{equation}
for a constant $C$ depending on the data. Choose a constant $\sigma_1\in(0,\sigma_3]$ so that if $ \om_1(b)+\om_2(\Sen, \msB_{en})+\om_3(\Phi_{bd},\pex)\le \sigma_1$ then \eqref{2d-2-b1} implies that $u>0$ and $\rho>0$ in $\ol{\N}$. Furthermore, $(\rho,{\bf u},p,\Phi)$ satisfy  \eqref{theorem1-est2} and solve the boundary value problem \eqref{2-a2} with \eqref{2-b9}--\eqref{2-b9-3}. This proves (a) of Theorem \ref{nf-theorem1}.

Next, we prove (b) of Theorem \ref{nf-theorem1}. Suppose that $(b,\Sen,\msB_{en},\Phi_{bd}, \pex)$ satisfies \eqref{theorem3-est3}. Let $(\rho^{(j)},\bu^{(j)}, p^{(j)}, \Phi^{(j)})$ ($j=1$, $2$)  be two solutions to \eqref{2-a2} with \eqref{2-b9}--\eqref{2-b9-3}.
Since ${\bf u}^{(j)}\in C_{(-\alpha, \Gamw)}^{1,\alp}(\N)$, the equation
\begin{equation}\label{psidiv}
\Delta\psi^{(j)}=\Div J{\bf u}^{(j)}
\quad \text{where}\quad J=\begin{pmatrix}
0&-1\\
1&\phantom{-}0
\end{pmatrix}
\end{equation}
with boundary conditions
\begin{equation}\label{2d-2-b2}
\der_{x_1}\psi^{(j)}=0\;\;\tx{on}\;\;\Gamen\cup \Gamex,\quad \psi^{(j)}=0\;\;\tx{on}\;\;\Gamw
\end{equation}
has a unique weak solution $\psi^{(j)}\in H^1(\N)$. Adjusting the proof of \cite[Theorem 3.8]{Ha-L} gives
$$\|\psi^{(j)}\|_{\alp,\N}\le C_1\|{\bf u}\|_{0,\N}.$$
Since $\N$ is a rectangle in $\R^2$, one has
$$\|\psi^{(j)}\|_{1,\alp,\N}\le C_2\|{\bf u}^{(j)}\|_{\alp,\N}$$
by  the method of reflection and uniqueness of weak solution to \eqref{psidiv}-\eqref{2d-2-b2}.
Finally, the Schauder estimate with scaling yields
$$\|\psi^{(j)}\|_{2,\alp,\N}^{(-1-\alp,\corners)}\le C_3\|{\bf u}^{(j)}\|_{1,\alp,\N}^{(-\alp,\corners)}.$$
Here, the estimate constants $C_{k}$ for $k=1,2,3$ depend only on $L$ and $\alp$. For more details, one can refer to \cite{BDX, BDX2, GilbargTrudinger, Ha-L} and references therein.

For each $j=1,2$, define a function $\vphi^{(j)}$ by
\begin{equation}
\label{2d-2-b3}
\vphi^{(j)}(x_1,x_2)=\int_0^{x_1} (u^{(j)}-\psi_{x_2}^{(j)})(t,x_2)\;dt\quad\tx{in}\quad\N.
\end{equation}
It follows from \eqref{2-b9-1}, \eqref{2d-2-b2}, and \eqref{2d-2-b3} that
\begin{equation*}
\nabla\vphi^{(j)}={\bf u}^{(j)}-\nabla^{\perp}\psi^{(j)}\quad\tx{in}\quad\N.
\end{equation*}
Define
\begin{equation*}
S^{(j)}=c_v \ln\left(\frac{ p^{(j)}}{\mfp \left(\rho^{(j)}\right)^{\gam}}\right),\quad \msK^{(j)}=\msB^{(j)}-\Phi^{(j)}
\quad\tx{in}\quad \N\,\, \text{for}\,\, j=1\,\,\text{and}\,\, 2.
\end{equation*}

Choose $\sigma_2\in(0,\sigma_1]$ small so that \eqref{theorem1-est2} implies $\rho^{(j)}>0$ and $u^{(j)}>0$ in $\ol{\N}$ for $j=1,2$. Then $\mfU^{(j)}= (\vphi^{(j)}, \psi^{(j)}, \Phi^{(j)}, S^{(j)}, \msK^{(j)})$ solves \eqref{nf-2-b6}--\eqref{eqSK} with \eqref{nf-3-a5}--\eqref{bcSK}. It is easy to check that each $\mfU^{(j)}$ satisfies the estimate \eqref{theorem3-est2}. Then one can reduce $\sigma_2$ on $(0,\min(\sigma_1,\sigma_4)]$ so that if \eqref{theorem1-est3} holds, then (b) of Theorem \ref{nf-theorem3} yields $\mfU^{(1)}=\mfU^{(2)}$. Hence $(\rho^{(1)}, \bu^{(1)}, p^{(1)}, \Phi^{(1)}) = (\rho^{(2)}, \bu^{(2)}, p^{(2)}, \Phi^{(2)})$. This finishes the proof for
(b) of Theorem \ref{nf-theorem1}.
\end{proof}

The rest of the paper is devoted to prove Theorem \ref{nf-theorem3}.

\subsection{Framework for proof of Theorem \ref{nf-theorem3}}\label{sectionlinear}
In order to prove Theorem \ref{nf-theorem3}, we use the method of iteration. For that purpose, we linearize \eqref{nf-2-b6}-\eqref{nf-2-b9} and the boundary condition \eqref{bc-vphi-ex}.

For ${\bf s}=(s_1,s_2)\in \R^2$, set ${\bf s}^{\perp}=(s_2,-s_1)$. For $(\varsigma, \eta, z)\in \R^3$, ${\bf q}=(q_1,q_2)\in \R^2$, and ${\bf s}\in\R^2$, define ${\bf A}=(A_1,A_2)$ and $B$ by
\begin{equation}
\label{nf-3-a6}
\begin{split}
&{A}_j(\varsigma, \eta, z,{\bf q},{\bf s})=B(\varsigma, \eta, z,{\bf q},{\bf s})q_j\quad\tx{for}\quad j=1,2\\
&B(\varsigma, \eta, z,{\bf q},{\bf s})=H(\varsigma, \eta+z-\frac{1}{2}|{\bf q}+{\bf s}^{\perp}|^2).
\end{split}
\end{equation}
It follows from \eqref{2d-2-a8} that ${\bf A}(\varsigma, \eta, z, {\bf q},\s)$ and $B(\varsigma, \eta, z, {\bf q},\s)$ are infinitely differentiable with respect to $\varsigma$, $\eta$, $z$, $\q$ and $\s$ if $\eta+z-\frac{1}{2}|{\bf q}+\s^\perp|^2>0$.
In terms of ${\bf A}$ and $B$,  \eqref{nf-2-b6} and \eqref{nf-2-b9} can be written as
\begin{align}
\label{nf-3-a7}
&\Div({\bf A}(S,\msK,\Phi,\nabla\vphi,\nabla\psi))=-\Div(B(S, \msK, \Phi,\nabla\vphi,\nabla\psi)\nabla^{\perp}\psi),\\
\label{nf-3-a8}
&\Delta \Phi=B(S, \msK,\Phi,\nabla\vphi,\nabla\psi)-b.
\end{align}

Set
$$
(\Psi,\phi):=(\Phi-\Phi_0,\vphi-\vphi_0)
$$
where $(\Phi_0, \vphi_0)$ is given by \eqref{2-b6}.
From now on, we introduce the following notations $\mclU :=(\Psi, \phi, \psi)$ and $\mclW: =(S, \msK)$. Denote  $\mclW_0: =(S_0, \msK_0)$
and $\mclWen: =(S_{en}, \msK_{en})$. We may denote $\partial_{x_i}$ by $\partial_i$. It is easy to see that there exists a $\delta_0\in (0, 1)$ such that if
\begin{equation}\label{defdelta}
|\mclW-\mclW_0|+|\Psi|+|\nabla\phi|+|\nabla\psi| \leq \delta_0,
\end{equation}
then ${\bf A}$ and $B$ are smooth with respect to their variables $(\varsigma, \eta, z,{\bf q},{\bf s})$ and the system \eqref{nf-3-a7} and \eqref{nf-3-a8} is a uniformly elliptic system.

Denote $\mfV= (S_0, \msK_0, \Phi_0, \nabla\varphi_0, {\bf 0})$.  Let $(\mclU, \mclW) =(\Psi, \phi, \psi, S, \msK)  \in (C^2(\N))^3\times (C^1(\N))^2$ be a solution of \eqref{nf-2-b6}--\eqref{eqSK} and \eqref{nf-3-a5}--\eqref{bcSK}. Then $(\Psi, \phi)$ satisfy the equations
\begin{equation}
\label{nf-3-c1}
\left\{
\begin{aligned}
&L_1(\Psi, \phi) =\Div {\bf F}(\rx, \mclW-\mclW_0, \Psi,  \nabla \phi, \nabla \psi)\\
&L_2(\Psi, \phi) =f_1(\rx, \mclW-\mclW_0, \Psi, \nabla \phi, \nabla \psi)
\end{aligned}
\right.
\end{equation}
where $L_1, L_2, {\bf F}=(F_1,F_2)$ and $f_1$ are defined as follows:
\begin{equation}
\label{2d-3-a2}
\begin{split}
&L_1(\Psi,\phi)=\sum_{i,j=1}^2\partial_i\Bigl(\der_{q_j}A_i(\mfV_0)\der_j\phi+\Psi\der_zA_i(\mfV_0)\Bigr),\\
&L_2(\Psi,\phi)=\Delta \Psi-\Bigl(\Psi\der_zB(\mfV_0)+\nabla\phi\cdot \der_{{\bf q}}B(\mfV_0)\Bigr),
\end{split}
\end{equation}
and
\begin{equation}
\label{nf-F}
\begin{split}
F_i(\rx,{Q})=&-\int_0^1D_{(\varsigma,\eta,{\bf s})}
A_i(\mfV_0+t{Q})\;dt\cdot({\varsigma}, {\eta},{\bf s})\\
&-\int_0^1D_{(z,{\bf q})}A_i(\mfV_0+t{Q})-D_{(z,{\bf q})}A_i(\mfV_0)\;dt
\cdot( z, {\bf q})-
B(\mfV_0+{Q})({\bf s}^{\perp})_i
\end{split}
\end{equation}
for $i=1$, $2$, and
\begin{equation}
\label{nf-f1}
\begin{split}
f_1(\rx,{Q})=&\int_0^1D_{(\varsigma,\eta,{\bf s})}
B(\mfV_0+t {Q})\;dt\cdot( {\varsigma},{\eta},{\bf s})-(b-b_0)\\
&+\int_0^1D_{(z,{\bf q})}B(\mfV_0+t {Q})-D_{(z,{\bf q})}B(\mfV_0)\;dt
\cdot(z, {\bf q})
\end{split}
\end{equation}
with
${Q}=({\varsigma},{\eta}, z, {\bf q}, {\bf{s}})\in\R^3\times(\R^2)^2$.

Note that $\mfU_0$ satisfies
\begin{equation*}
\mfp \exp(\frac{S_0}{c_v}) H^\gam(S_0, \msK_0+\Phi_0 -\frac{1}{2} |\nabla\vphi_0|^2)=\bar p(L,x_2)\;\;\tx{on}\;\;\Gamex.
\end{equation*}
Subtracting this from \eqref{bc-vphi-ex} yields
\begin{equation}
\label{3-d3}
\partial_1 \phi=g(\rx, \mclW-\mclW_0, \nabla\phi, \nabla\psi) \;\;\tx{on}\;\;\Gamex
\end{equation}
for $g$ defined by
\begin{equation}
\label{defg}
\begin{aligned}
g(\rx,{\varsigma},{\eta},{\bf q}, {\bf s})
=&-s_2+  \frac{{\eta}+\Psi_{bd}- \frac 12|{\bf q}+{\bf s}^{\perp}|^2}{\bar u(L)}\\
& +\gam\frac{\pex^{\frac{\gamma-1}{\gam}}(\mfp \exp(\frac{S_0+\varsigma}{c_v}))^{\frac{1}{\gam}}- (\bar p (L))^{\frac{\gamma-1}{\gam}}( \mfp \exp(\frac{S_0}{c_v}))^{\frac{1}{\gam}}}{(\gam-1)\bar u(L)} .
\end{aligned}
\end{equation}

Since $\mclW_0$ is a constant vector, one has $\der_{2}\mclW=\der_2(\mclW-\mclW_0)$ so that \eqref{nf-2-b7} can be rewritten as
\begin{equation}
\label{nf-3-b2}
\Delta \psi=f_2(\rx,\mclW-\mclW_0,\Psi,\nabla \phi,\nabla\psi,\der_2(\mclW-\mclW_0))
\end{equation}
with $f_2$ defined by
\begin{equation}
\label{nf-3-b5}
f_2(\rx,{Q},{\xi},{\tau})
=-\frac{T(B(\mfV_0+{Q}),S_0+{\varsigma}) {\xi}-{\tau}}
{\der_1\vphi_0(\rx)+{q}_1+{s}_2}.
\end{equation}

We have shown that the boundary value problem \eqref{nf-2-b6}--\eqref{eqSK} with boundary condition \eqref{nf-3-a5}--\eqref{bcSK} is equivalent to \eqref{nf-3-c1}, \eqref{nf-3-b2} and \eqref{eqSK} in $\N$ with boundary conditions \eqref{s-bc-psi}, \eqref{bcSK}, \eqref{3-d3}, and
\begin{align}
\label{2d-3-b2}
\phi=0\,\,\tx{on}\;\;\Gamen,\quad \der_{{\bf n}_w}\phi=0\,\,\tx{on}\;\;\Gamw
\end{align}
and
\begin{align}\label{2d-3-bcPsi}
\Psi=\Psi_{bd}\,\,\tx{on}
\;\;\Gamen\cup\Gamex,\quad \text{and}\quad
\der_{{\bf n}_w}\Psi=0\,\,\tx{on}\;\;\Gamw,
\end{align}
where $\Psi_{bd}=\Phi_{bd}-\Phi_0$.
 To prove Theorem \ref{nf-theorem3}, it suffices to prove that the nonlinear system  \eqref{nf-3-c1}, \eqref{nf-3-b2} and \eqref{eqSK} in $\N$ with boundary conditions \eqref{s-bc-psi}, \eqref{bcSK}, \eqref{3-d3}, \eqref{2d-3-b2} and \eqref{2d-3-bcPsi} has unique solution $(\mclU, \mclW)$ when $\sigma_3$ and $\sigma_4$ are chosen sufficiently small.

To prove unique solvability of this nonlinear boundary value problem, we take the following steps:

Step 1. For a fixed constant $\alpha \in (0, 1)$, define
\begin{equation}
\label{2d-4-e1}
\mcl{P}(M)=\{\mclW \in [C^{1,\frac{\alp}{2}}(\ol{\N})]^2:
\|\mclW-\mclW_0\|_{1,\alp,\N}\le M \sigma\}
\end{equation}
where $\sigma =\om_1(b)+\om_2(\Sen,\msB_{en})+\om_3(\Phi_{bd},\pex)$ and  the constant $M>0$ is to be determined later such that $M\sigma\leq \delta_0/2$.
Fix $\mclW^*= (S^*,\msK^*)\in \mcl{P}(M)$, and solve a nonlinear problem
\begin{align}
\label{2d-4-a2}
&\begin{cases}
L_1(\Psi, \phi) =\Div {\bf F}(\rx, \mclW^*-\mclW_0, \Psi,  \nabla \phi, \nabla \psi)\\
L_2(\Psi, \phi) =f_1(\rx, \mclW^*-\mclW_0, \Psi, \nabla \phi, \nabla \psi)\\
\end{cases}\;\;\tx{in}\;\;\N
\end{align}
and
\begin{equation}\label{2d-4-a5}
\Delta\psi=f_2(\rx, \mclW^*-\mclW_0 ,\Psi,\nabla \phi,\nabla\psi,\der_2 \mclW^*)\,\,\tx{in}\;\;\N
\end{equation}
with boundary conditions \eqref{s-bc-psi}, \eqref{bcSK},  \eqref{2d-3-b2}, \eqref{2d-3-bcPsi}, and
\begin{equation}
\label{2d-4-a3}
\der_1\phi=g(\rx, \mclW^*-\mclW_0, \Psi, \nabla\phi, \nabla\psi)\,\,\tx{on}\;\;\Gamex.
\end{equation}

We have the following proposition for the boundary value problem \eqref{2d-4-a2}--\eqref{2d-4-a5} with boundary conditions \eqref{s-bc-psi}, \eqref{bcSK},  \eqref{2d-3-b2}, \eqref{2d-3-bcPsi}, and \eqref{2d-4-a3}.
\begin{proposition}
\label{proposition-1}
Suppose that $(\bar{\rho}, \bar{\bf u}, \bar p, \Phi_0)$ is the subsonic background solution associated with the parameters $b_0>0$, $S_0>0$, $J_0>0$, $\rho_0>\rhos$, $E_0$, and $L$. Assume that
 $\Phi_{bd}$ satisfies the compatibility condition \eqref{compt}.
Then, there exist a constant $\sigma_5>0$ depending on the data, $\alp$ and $M$ so that if
    \begin{equation}
   \label{2d-3-c1}
  \om_1(b)+\om_2(\Sen,\msB_{en})+\om_3(\Phi_{bd},\pex)\le \sigma_5,
\end{equation}
then the boundary value problem \eqref{2d-4-a2}--\eqref{2d-4-a5} with boundary conditions \eqref{s-bc-psi}, \eqref{bcSK},  \eqref{2d-3-b2}, \eqref{2d-3-bcPsi}, and \eqref{2d-4-a3} has a unique solution $\mclU= (\Psi,\phi, \psi)$ satisfying
\begin{equation}
\label{2d-3-c2}
\begin{split}
\|\mclU\|_{2,\alp,\N}^{(-1-\alp,\corners)}
\le C\left(\om_1(b)+\om_2(\Sen,\msB_{en})+\om_3(\Phi_{bd},\pex)\right)
\end{split}
\end{equation}
where the constant $C$ depends only on the data and $\alp$.
\end{proposition}
The proof of Proposition \ref{proposition-1} is given in Section \ref{section-nlbvp}.

Step 2.
Let $\mclU = (\Psi,\phi, \psi)\in (C^1(\ol{\N})\cap C^2(\N))^3$ be unique solution of \eqref{2d-4-a2}--\eqref{2d-4-a5} with boundary conditions \eqref{s-bc-psi}, \eqref{bcSK},  \eqref{2d-3-b2}, \eqref{2d-3-bcPsi}, and \eqref{2d-4-a3} associated with $\mclW^* =(S^*, \msK^*)\in\mcl{P}(M)$. It follows from \eqref{s-bc-psi}, \eqref{2d-3-b2}, and \eqref{2d-4-a2} that the vector field
\begin{equation}
\label{v}
V=H(S^*, \msK^*+\Phi_0+ \Psi-\frac 12|\nabla\vphi_0+\nabla\phi+\nabla^{\perp}\psi|^2)
(\nabla\vphi_0+\nabla\phi+\nabla^{\perp}\psi)
\end{equation}
satisfies
\begin{equation}
\label{2d-3-b6}
\Div V=0\quad \tx{in}\quad \N,\quad V\cdot{\bf n}_w=0\quad\tx{on}\quad\Gamw.
\end{equation}
The following lemma guarantees that there exists a solutions $\mclW$
 of the  problem
\begin{equation}
\label{2d-3-b4}
\begin{split}
&V\cdot\nabla \mclW=0\quad\tx{in}\quad \N,\quad  \mclW=\mclWen\quad\tx{on}\quad \Gamen.
\end{split}
\end{equation}
\begin{lemma}
\label{nf-lemma-4-1}
Suppose that a vector field ${\bf V}=(V_1,V_2)$ satisfies \eqref{2d-3-b6} and the estimate
\begin{equation}
\label{2d-3-b7}
\|{\bf V}\|_{1,\alp,\N}^{(-\alp,\corners)}\le K_0
\end{equation}
for a constant $K_0>0$. In addition, assume that there exists a constant $\nu^*>0$ satisfying
\begin{equation}
\label{2d-3-b8}
V_1\ge \nu^* \quad\tx{in}\quad\ol{\N}.
\end{equation}
Then \eqref{2d-3-b4} has a unique solution $\mclW \in (C^{1, \alp}(\ol{\N}))^2$ satisfying
\begin{equation}
\label{nf-4-c3}
\begin{split}
\|\mclW-\mclW_0\|_{1, \alpha,\N}\le C^*\|\mclWen-\mclW_0\|_{1,\alp,\Gamen},
\end{split}
\end{equation}
where the constant $C^*$ depends only on $L$, $\nu^*$,  $K_0$ and  $\alp$.
\end{lemma}

\begin{proof}
Set
\begin{equation}
\label{w}
w(x_1,x_2):=\int_0^{x_2}V_1(x_1,y)\;dy\quad\tx{in}\quad \N.
\end{equation}
In fact, $w$ is the stream function of the flow.
It follows from \eqref{2d-3-b4}, \eqref{2d-3-b7}, and \eqref{2d-3-b8} that $w$ satisfies
\begin{equation}
\label{2d-3-c5}
\nabla^{\perp}w={\bf V}\;\;\;\tx{and}\;\; \;\der_2 w\ge \nu^*\;\;\tx{in}\;\;\N,\quad
\der_1 w=0\;\;\tx{on}\;\;\Gamw
\end{equation}
and
\begin{equation}\label{3-estw}
\|w\|_{2, \alpha, \N}^{(-1-\alpha, \Gamw)}\leq CK_0
\end{equation}
with the constant $C$ depending only on $L$.
Then the implicit function theorem implies that for any $x_1\in[0,L]$ and $\lambda\in[w(0,0), w(0,1)]$, there exists a unique $h(x_1,\lambda)\in[0,1]$ satisfying
\begin{equation*}
w(x_1,h(x_1,\lambda))=\lambda.
\end{equation*}
Moreover, such $h$ is continuously differentiable with respect to $x_1\in[0,L]$ and $\lambda\in [w(0,0), w(0,1)]$. For any given $\lambda\in [w(0,0), w(0,1)]$, note that $x_2=h(x_1,\lambda)$ is a $C^1$ curve crossing $x_1=0$ and $x_1=L$ exactly once. Therefore, for any $\rx=(x_1,x_2)\in\ol{\N}$, one can find a unique $\vartheta\in[0,1]$ satisfying
\begin{equation}
\label{2d-3-c4}
w(x_1,x_2)=w(0,\vartheta).
\end{equation}
It follows from \eqref{2d-3-c5} and \eqref{2d-3-c4} that if $\mclW$ solves \eqref{2d-3-b4}, then one has
\begin{equation}
\label{2d-3-c3}
\mclW(\rx)=\mclWen(0,\vartheta)
\end{equation}
Note that $\mcl{G}(\vartheta)=w(0,\vartheta)$ is an invertible function from $[0,1]$ onto $[w(0,0), w(0,1)]$. Set
\begin{equation}
\label{L}
\mathscr{L}(x_1,x_2)=\mcl{G}^{-1}\circ w(x_1,x_2).
\end{equation}
Then \eqref{2d-3-c3} implies that $\mclW$ given by
\begin{equation}
\label{2d-3-c6}
\mclW(\rx)=\mclWen\circ \mathscr{L}(\rx)
\end{equation}
solves \eqref{2d-3-b4}. Note that
\begin{equation}\label{gradL}
\nabla\mathscr{L}(\rx)=\frac{\nabla w(\rx)}{\der_2 w(0,\mathscr{L}(\rx))}=-\frac{{\bf V}^\perp(\rx)}{V_1(0,\mathscr{L}(\rx))}.
\end{equation}
Thus
\begin{equation}\label{estL}
\|\mathscr{L}\|_{1, \alpha, \N}\leq C \|{\bf V}\|_{\alpha, \N}.
\end{equation}
Combining this with \eqref{3-estw} and \eqref{2d-3-c6} gives that $\mclW$ defined by \eqref{2d-3-c6} satisfies the estimate \eqref{nf-4-c3} because $\mclW_0$ is a constant vector.

Let $\mclW^{(1)}$ and $\mclW^{(2)}$  be two solutions of \eqref{2d-3-b4} satisfying \eqref{nf-4-c3}.
Then $\hat\mclW=\mclW^{(1)}-\mclW^{(2)}$ solves
\begin{equation*}
{\bf V}\cdot\nabla \hat\mclW=0\;\;\tx{in}\;\;\N,\quad
\hat\mclW=0\;\;\tx{on}\;\;\Gamen.
\end{equation*}
It follows from \eqref{2d-3-c5} that $\hat\mclW$ is a constant along each level curve $x_2=h(x_1,\lambda)$ of $w$  given by \eqref{w}. For any given $\rx=(x_1,x_2)$, there exists unique $\lambda\in[w(0,0),w(0,1)]$ such that $x_2=h(x_1,\lambda)$. The level curve $x_2=h(x_1,\lambda)$ crosses $\Gamen$ on which $\hat\mclW=0$. From $\hat\mclW(\rx)=\hat\mclW(0,h(0,\lambda))=0$, we conclude that $\hat\mclW=0$ in $\N$, and this implies $\mclW^{(1)}=\mclW^{(2)}$ in $\N$ because $\rx$ is arbitrary in $\N$. This proves the uniqueness of a solution to \eqref{2d-3-b4}.
\end{proof}

Step 3. Define an iteration mapping $\mcl{J}:\mcl{P}(M)\to (C^{1, \alpha}(\bar\N))^2$ by
\begin{equation}
\label{2d-3-b5}
\mcl{J}\mclW^*= \mclW,
\end{equation}
where $\mclW$ is the solution of \eqref{2d-3-b4} associated with ${\bf V}$ defined in \eqref{v}.
And, choose positive constants $M$ and $\sigma$ so that the mapping $\mcl{J}$ defined by \eqref{2d-3-b5} maps $\mcl{P}(M)$ into itself. In Section \ref{section-pf-theorem2}, we also show that $\mcl{J}$ is a continuous map in $C^{1, \alpha/2}(\bar\N)$ so that $\mcl{J}$ has a fixed point in $\mcl{P}(M)$. This will prove (a) of Theorem \ref{nf-theorem3}.

Step 4. Using the stream function formulation in Step 2 and the estimate for the difference between solutions for the problem \eqref{nf-2-b6}--\eqref{eqSK} with \eqref{nf-3-a5}--\eqref{bcSK} in a weaker space $C^{1, \beta}(\bar\N)$ ($\beta<\alpha$) yields the uniqueness  of the solutions for the associated problem. The detail is also given in Section \ref{section-pf-theorem2}.

\section{Proof of Proposition \ref{proposition-1}}
\label{section-nlbvp}
In this section, we prove Proposition \ref{proposition-1} by the  iteration method. We first study a linear boundary value problem.
\subsection{Boundary value problem for the linear system}
Suppose that ${\mfF}=(\mcF_1,\mcF_2)\in (C^{1,\alp}_{(-\alp,\corners)}(\N))^2$, $\mff_1\in C^{\alp}(\ol{\N})$,  and $\mfg\in C^{1,\alp}_{(-\alp,\der\Gamex)}(\Gamex)$.  We first consider the following linear system
\begin{align}
\label{2d-4-b1}
&\begin{cases}
L_1(\Psi,\phi)=\Div {\bf \mfF}\\
L_2(\Psi,\phi)=\mff_1
\end{cases}\quad\tx{in}\quad \N
\end{align}
with boundary conditions
\begin{align}
\label{2d-4-b2}
\phi=0\,\,\tx{on}\;\;\Gamen,\quad
\der_2\phi=0\,\,\tx{on}\;\;\Gamw,\quad
\der_1\phi=\mfg\,\,\tx{on}\;\;\Gamex
\end{align}
and
\begin{align}\label{linearbcPsi}
\Psi=\Psi_{bd}\,\,\tx{on}\;\;\Gamen\cup\Gamex\quad \text{and}\quad
\der_2\Psi=0\,\,\tx{on}\;\;\Gamw.
\end{align}

To prove well-posedness of \eqref{2d-4-b1}-\eqref{linearbcPsi}, we need to take a closer look at the linear operators $L_1$ and $L_2$ defined by \eqref{2d-3-a2}. Let $(\bar{\rho}, \bar\bu, \bar p, \Phi_0)$ be the subsonic background solution associated with the parameters $b_0>0$, $S_0>0$, $J_0>0$, $\rho_0>\rhos$, $E_0$ and $L$. Set
\begin{equation}
\label{2d-4-b5}
\begin{split}
\mfraka_{ij}(\rx)=\der_{q_j}A_i(\mfV_0),
\quad
\mfrakb_i(\rx)=\der_{z}A_i(\mfV_0)
\quad \mfrakc_i(\rx)=\der_{q_i}B(\mfV_0),
\quad\;\;\mfrakd(\rx)=\der_{z}B(\mfV_0)
\end{split}
\end{equation}
for $i,j=1,2$, where ${\bf A}=(A_1,A_2)$ and $B$ are given by \eqref{nf-3-a6}. Then we have the following lemma.
\begin{lemma}
\label{nf-lemma-4-2}
Let $\mfraka_{ij}, \mfrakb_i, \mfrakc_i, \mfrakd$ be defined by \eqref{2d-4-b5}.
\begin{itemize}
\item[(a)]
The matrix $[\mfraka_{ij}(\rx)]_{i,j=1}^2$ is strictly positive and diagonal in $\N$, and there exits a constant $\nu_1>0$ satisfying
    \begin{equation}
    \label{nf-4-a6}
    \nu_1 { I_2}\le [\mfraka_{ij}(\rx)]_{i,j=1}^2\le \frac{1}{\nu_1}{ I_2}\quad\tx{for all}\quad \rx\in \N,
    \end{equation}
    where the constant $\nu_1$ depends only on the data;
\item[(b)] For each $k\in\mathbb{Z}_+$, there exists a constant $\mcl{C}_k>0$ depending on the data and $k$ such that
   \begin{equation}
   \label{2d-4-b6}
   \sum_{i,j=1}^2\|\mfraka_{ij}\|_{k,\N}
   +\sum_{i=1}^2\left(\|\mfrakb_i\|_{k,\N}+\|\mfrakc_i\|_{k,\N}\right)
   +\|\mfrakd\|_{k,\N}\le \mcl C_k;
    \end{equation}
\item[(c)] For each $i=1$ and $2$, we have
\begin{equation}
\label{nf-4-a7}
\mfrakb_i(\rx)+\mfrakc_i(\rx)=0\quad\tx{in}\quad \N;
\end{equation}
\item[(d)] There exists a constant $\nu_2>0$ depending only on the data such that
    \begin{equation}
    \label{2d-4-b7}
    \mfrakd(\rx)\ge \nu_2\quad\tx{in}\quad \N.
    \end{equation}
\end{itemize}
\begin{proof}
It follows from \eqref{2-b6} and \eqref{2d-2-a8} that $[\mfraka_{ij}]_{i,j=1}^2$ is a diagonal matrix with
\begin{equation}\label{nf-4-a9}
\mfraka_{11}= \bar{\rho}^{2-\gam}\left(1- \frac{\bar u^2}{\gam \mfp \exp (\frac{S_0}{c_v})\bar\rho^{\gam-1}}\right)
\quad \tx{and}\;\;\mfraka_{22}=
\bar{\rho}.
\end{equation}
And, \eqref{nf-4-a6} follows from \eqref{2d-1-a2} and \eqref{2-a7}. This proves (a). (b) easily follows from the smoothness of $(\bar{\rho}, \bar \bu, \bar p, \Phi_0)$. Direct computations for \eqref{2d-2-a8} and \eqref{nf-3-a6} give
\begin{equation*}
\mfrakb_i=-\mfrakc_i=\begin{cases}
\bar{\rho}^{2-\gam}\frac{\bar u}{\gam \mfp \exp (\frac{S_0}{c_v})}&\tx{for}\;\;i=1\\
0&\tx{for}\;\;i=2
\end{cases}\quad\tx{and}\quad
\mfrakd=\frac{\bar{\rho}^{2-\gam}}{\gam \mfp \exp (\frac{S_0}{c_v})}.
\end{equation*}
Then (c) is proved, and \eqref{2-a7} implies (d).
\end{proof}
\end{lemma}

 Lemma \ref{nf-lemma-4-2} plays a crucial role in proving the existence of solution for the problem \eqref{2d-4-b1} and \eqref{2d-4-b2}. More precise statement is given in the following lemma.
\begin{lemma}
\label{2d-lemma-2}
Suppose that ${\mfF}=(\mcF_1,\mcF_2)\in (C^{1,\alp}_{(-\alp,\corners)}(\N))^2$, $\mff_1\in C^{\alp}(\ol{\N})$,  and $\mfg\in C^{1,\alp}_{(-\alp,\der\Gamex)}(\Gamex)$ for $\alpha\in (0,1)$. If, in addition,  $\Psi_{bd}$ satisfies the compatibility condition
\begin{equation}
\label{2d-4-c2}
\der_2\Psi_{bd}=0\quad\tx{on}\quad (\ol{\Gamen}\cap \ol{\Gamex})\cap \ol{\Gamw}.
\end{equation}
Then the linear boundary value problem \eqref{2d-4-b1}-\eqref{linearbcPsi} has a unique solution $(\phi, \Psi)\in (C^{1,\alp}(\ol{\N})\cap C^{2,\alp}(\N))^2$. Moreover, $(\phi,\Psi)$ satisfy the estimate
\begin{equation}
\label{2d-4-c1}
\begin{split}
&\|(\phi,\Psi)\|_{2,\alp,\N}^{(-1-\alp,\corners)}\le C_1^\sharp(\|\mfg\|_{1,\alp,\Gamex}^{(-\alp,\der\Gamex)}+\|\Psi_{bd}\|_{2,\alp,\N}^{(-1-\alp,\corners)}
+
\|\mfF\|_{1,\alp,\N}^{(-\alp,\corners)}+\|\mff_1\|_{\alp,\N})
\end{split}
\end{equation}
for a constant $C_1^\sharp>0$ depending only on the data and $\alp$.
\end{lemma}
\begin{proof}
The lemma is proved in three steps. Steps 2-3 are quite similar to that for \cite[Proposition 4.1]{BDX}. We give a brief sketch for these steps.  One can refer to \cite{BDX} for details.

Step 1. Define
\begin{equation*}
\Psi_{bd}^*(x_1,x_2)=\chi(x_1)\Psi_{bd}(0,x_2)+(1-\chi(x_1))\Psi_{bd}(L,x_2)
\end{equation*}
for a smooth function $\chi(x_1)$ satisfying
$\chi'(x_1)\le 0$, $|\chi'(x_1)|\le \frac{6}{L}$ and
$$
\chi(x_1)= 1\,\, \tx{for}\;\;x_1\le \frac L3\quad \text{and}\quad \chi(x_1)=0\,\,\tx{for}\;\;x_1\ge \frac{2L}{3}.
$$
It follows from \eqref{2d-4-c2} that $\Psi_{bd}^*$ satisfies
\begin{equation*}
\der_2\Psi_{bd}^*=0\quad\tx{on}\quad\Gam_w.
\end{equation*}
$(\phi,\Psi)$ solves \eqref{2d-4-b1}-\eqref{linearbcPsi} if and only if $(\phi,\hat{\Psi}):=(\phi,\Psi-\Psi_{bd}^*)$ satisfies
\begin{equation}
\label{2d-4-c3}
\begin{cases}
L_1(\hat{\Psi},\phi)=\Div({\mfF}-{\bm\mfrakb}\Psi^*_{bd})=:\Div{\mfF}^*\\
L_2(\hat{\Psi},\phi)=(\mff_1-\mfrakd\Psi^*_{bd})-\Div(\nabla\Psi^*_{bd})=:\mff_1^*+\Div{\mfG}^*
\end{cases}\quad\tx{in}\quad\N
\end{equation}
with boundary conditions \eqref{2d-4-b2} and
\begin{equation}\label{bchPsi}
\hat\Psi=0\,\,\tx{on}\;\;\Gamen\cup\Gamex\quad\text{and}\quad
\der_2\hat\Psi=0\,\,\tx{on}\;\;\Gamw.
\end{equation}

Define  $\mcl{H}= \{(\zeta,\omega)\in [H^1(\N)]^2: \zeta=0\;\tx{on}\;\Gamen, \omega=0\;\tx{on}\;\Gamen\cup\Gamex\}$. If $(\phi, \hat\Psi)\in \mcl{H}$ satisfies
\begin{equation}
\label{2d-4-c4}
\mathfrak{L}[(\phi,\hat\Psi),(\zeta,\omega)]=\langle({\mfF}^*,\mfg,\mff_1^*, {\mfG}^*),(\zeta,\omega)\rangle
\end{equation}
for all $(\zeta,\omega)\in \mcl{H}$ where
\begin{equation*}
\mathfrak{L}[(\phi,\hat{\Psi}),(\zeta,\omega)]
=\int_{\N} \sum_{i=1}^2(\mfraka_{ii}\der_i\phi+\mfrakb_i\hat{\Psi})\der_i\zeta
+\nabla\hat{\Psi}\cdot\nabla\omega+(\mfrakc\cdot\nabla\phi+\mfrakd\hat{\Psi})\omega\;d\rx
\end{equation*}
and
\begin{equation*}
\begin{split}
\langle({\mfF}^*,\mfg,\mff_1^*, {\mfG}^*),(\zeta,\omega)\rangle
=&\int_{\N}{\mfF}^*\cdot\nabla\zeta+{\mfG}^*\cdot\nabla\omega-\mff_1^*\omega\;d\rx\\
&-\int_{\der\N}({\mfF}^*\cdot{\bf n}_{out})\zeta+({\mfG}^*\cdot{\bf n}_{out})\omega \;ds+
\int_{\Gamex}\mfraka_{11}\mfg\zeta \;dx_2
\end{split}
\end{equation*}
with ${\bf n}_{out}$ the outward unit normal of $\der\N$, then we call $(\phi, \hat\Psi)$ the weak solution of the problem \eqref{2d-4-c3}, \eqref{2d-4-b2}, and \eqref{bchPsi}. It is easy to see that the classical solution of \eqref{2d-4-c3}, \eqref{2d-4-b2}, and \eqref{bchPsi} must be a weak solution.

Using (a), (c) and (d) of Lemma \ref{nf-lemma-4-2} and Poincar\'{e} inequality yields that there exists a constant $\nu_3>0$ depending only on the data to satisfy
\begin{equation*}
\begin{split}
\mathfrak{L}[(\zeta,\omega),(\zeta,\omega)]
\ge \int_{\N}\nu_1|\nabla\zeta|^2+|\nabla\omega|^2+\nu_2\omega^2\;d\rx
\ge {\nu_3}\left(\|\zeta\|^2_{H^1(\N)}+\|\omega\|^2_{H^1(\N)}\right)
\end{split}
\end{equation*}
for all $(\zeta,\omega)\in \mcl{H}$. This implies that the bilinear operator $\mathfrak{L}:\mcl{H}\times\mcl{H}\to \R$ is coercive. It is easy to see that
\begin{equation}
|\mathfrak{L}[(\phi, \hat\Psi),(\zeta,\omega)]|\leq C (\|\phi\|_{H^1(\N)}+ \|\hat \Psi\|_{H^1(\N)}) (\|\zeta\|_{H^1(\N)}+ \|\omega\|_{H^1(\N)}).
\end{equation}
This means that $\mathfrak{L}$ is a bounded bilinear functional on $\mcl{H}\times \mcl{H}$.
Furthermore, it follows from the trace inequality and H\"{o}lder inequalities that one can find a constant $\mcl{C}$ depending only on the data to satisfy
\begin{equation*}
\begin{split}
&|\langle({\mfF}^*,\mfg,\mff_1^*, {\mfG}^*),(\zeta,\omega)\rangle|\\
&\le \mcl{C} (\|{\mfF}^*\|_{L^{\infty}(\N)}+\|{\mfG}^*\|_{L^{\infty}(\N)}
+\|\mff_1^*\|_{L^{\infty}(\N)} +\|\mfg\|_{L^{\infty}(\Gamex)})\left(\|\zeta\|_{H^1(\N)}+\|\omega\|_{H^1(\N)}\right)
\end{split}
\end{equation*}
for all $(\zeta,\omega)\in \mcl{H}$. From now on, the constant $\mcl{C}$  depends only on the data and $\alp$, which may vary from line to line.
Then Lax-Milgram theorem and Cauchy-Schwartz inequalities imply that there exists unique $(\phi,\hat{\Psi})\in\mcl{H}$ satisfying \eqref{2d-4-c4} and the estimate
\begin{equation}
\label{3-c4}
\|\phi\|_{H^1(\N)}+\|\hat{\Psi}\|_{H^1(\N)}\le
\mcl{C}(\|{\mfF}^*\|_{L^{\infty}(\N)}+\|{\mfG}^*\|_{L^{\infty}(\N)}
+\|\mff_1^*\|_{L^{\infty}(\N)} +\|\mfg\|_{L^{\infty}(\Gamex)}).
\end{equation}

Step 2.  Combining \eqref{3-c4}, H\"{o}lder inequality, Sobolev inequality and Poincar\'{e} inequality gives that
the weak solution $(\phi,\hat{\Psi}) \in\mcl{H}$ satisfies
\begin{equation}
\label{2d-4-c5}
\|\phi\|_{\alp,\N}+\|\hat{\Psi}\|_{\alp,\N}\le
\mcl{C}(\|{\mfF}^*\|_{L^{\infty}(\N)}+\|{\mfG}^*\|_{L^{\infty}(\N)}
+\|\mff_1^*\|_{L^{\infty}(\N)} +\|\mfg\|_{L^{\infty}(\Gamex)}).
\end{equation}

Substitute $\omega=0$ into \eqref{2d-4-c4} and regard $\phi$ as a weak solution of
\begin{equation*}
\int_{\N}\sum_{i=1}^2\mfraka_{ii}\der_i\phi\der_i \zeta=
\int_{\N}({\mfF}^*-\mfrakb \hat{\Psi})\cdot\nabla\zeta \;d\rx
-\int_{\N}({\mfF}^*\cdot{\bf n}_{out}\zeta)\;ds+\int_{\Gamex}\mfraka_{11}\mfg\zeta\;dy.
\end{equation*}
Then it follows from   \eqref{3-c4}, \eqref{2d-4-c5} and method of reflection with respect to $\Gamw$ that one has
\begin{equation}
\label{2d-4-c6}
\|\phi\|_{1,\alp,\N}\le \mcl{C}(\|{\mfF}^*\|_{\alp,\N}+\|\hat{\Psi}\|_{\alp,\N}+\|\mfg\|_{\alp,\Gamex}).
\end{equation}
 Once we have \eqref{2d-4-c6}, the equation for $\Psi$ can be regarded as an elliptic equation of divergence form. Using the compatibility condition \eqref{2d-4-c2} and the method of reflection with respect to $\Gamw$ yields
\begin{equation}
\label{2d-4-c7}
\|\Psi\|_{1,\alp,\N}\le \mcl{C}(\|{\mfG}^*\|_{\alp,\N}+\|\mff_1^*\|_{0,\N}+\|\phi\|_{1,\alp,\N}).
\end{equation}

Step 3. The scaling argument and Schauder estimate, together with \eqref{2d-4-c6}, \eqref{2d-4-c7}, give
\begin{equation}
\label{2d-4-c8}
\|\phi\|_{2,\alp,\N}^{(-1-\alp,\corners)}\le \mcl{C}(\|\phi\|_{1,\alp,\N}+\|{\mfF}^*\|_{1,\alp,\N}^{(-\alp,\corners)}+\|\mfg\|_{1,\alp,\Gamex}^{(-\alp,\der\Gamex)})
\end{equation}
and
\begin{equation}\label{2d-4-c81}
\|\hat{\Psi}\|_{2,\alp,\N}^{(-1-\alp,\corners)}\le \mcl{C}(\|\phi\|_{2,\alp,\N}^{(-1-\alp,\corners)}+\|{\mfG}^*\|_{1,\alp,\N}^{(-\alp,\corners)}+\|\mff_1^*\|_{\alp,\N}).
\end{equation}
It follows from \eqref{2d-4-c8}-\eqref{2d-4-c81}  that the linear boundary value problem \eqref{2d-4-c3} is uniquely solvable in $[C^1(\ol{N})\cap C^2(\N)]^2$. Therefore,   the boundary value problem \eqref{2d-4-b1} and \eqref{2d-4-b2} is uniquely solvable in the same space. Finally, the estimate \eqref{2d-4-c1} easily follows from \eqref{2d-4-c8} and \eqref{2d-4-c81}.
\end{proof}

\begin{remark}
Since $\N$ is a rectangle, one can use the reflection with respect to $\Gamw$ to get estimates \eqref{2d-4-c6} and \eqref{2d-4-c7}, which is simpler than that for the multidimensional domain case in \cite{BDX}.
\end{remark}

\subsection{Nonlinear boundary value problem}
Now we are in position to prove Proposition \ref{proposition-1}
by Lemma \ref{2d-lemma-2} and the contraction mapping principle.
\begin{proof}
[Proof of Proposition \ref{proposition-1}]
Define
\begin{equation*}
\begin{aligned}\label{nf-4-a1}
\mcl{K}(M_1):=\Big\{\mclU =(\Psi, \phi, \psi) \!\in \!(C^{2,\alp}_{(-1-\alpha, \Gamw)}({\N}))^3: &  \|\mclU\|_{2,\alp,\N}^{(-1-\alp,\corners)} \!\!\le\! M_1\sigma, \\
 &\phi=0\,\, \text{on}\,\,\Gamen,\,\, \psi =0 \,\,\text{on}\,\, \partial\N\setminus \Gamen\Big\}
\end{aligned}
\end{equation*}
where  $\sigma =\om_1(b)+\om_2(\Sen,\msB_{en})+\om_3(\Phi_{bd},\pex)$
and $M_1$ is a positive constant to be determined later satisfying $M_1\sigma\leq \delta_0/2$. Since $C_{(-1-\alpha, \Gamw)}^{2, \alpha}(\N)$ is a Banach space, it is easy to see that $\mcl{K}(M_1)$ is a closed set of the Banach space.

Note that if $\mclW^*\in \mcl{P}(M)$ and $(z, {\bf q}, {\bf s})$ satisfies
$|(z, {\bf q}, {\bf s})|\le \delta_0/2$, then it follows from \eqref{defdelta} that $F_i(\rx,\mclW^*-\mclW_0, z, {\bf q}, {\bf s})$(for $i=1,2$), $f_1(\rx,\mclW^*-\mclW_0, z, {\bf q}, {\bf s})$, $f_2(\rx,\mclW^*-\mclW_0, z, {\bf q}, {\bf s},\der_2 \mclW^*)$ and $g(\rx,\mclW^*-\mclW_0, {\bf q},{\bf s},\pex)$ are smooth with respect to $(z, {\bf q}, {\bf s})$, and they satisfy
\begin{equation}
\label{2d-4-d1}
\begin{split}
&|(f_1,F_1,F_2)(\rx,\mclW^*-\mclW_0, z, {\bf q}, {\bf s})|,\\
&\phantom{aaaa}\le C\left(M\sigma+|{\bf s}|+(M\sigma+ |{\bf s}|+| z|+|{\bf q}|)(| z|+|{\bf q}|)\right),\\
&|D_{(z,{\bf q})}(f_1,F_1,F_2)(\rx,\mclW^*-\mclW_0, z, {\bf q}, {\bf s})|\le C(M\sigma+|{\bf s}|+|{z}|+|{\bf q}|),\\
&|g(\rx,\mclW^*-\mclW_0, {\bf q}, {\bf s})|
\le C((M+1)\sigma+|{\bf q}|^2+|{\bf s}|),\\
&|D_{{\bf q}}g(\rx,\mclW^*-\mclW_0, {\bf q},{\bf s})|
\le C(|{\bf q}|+|{\bf s}|),\\
&|D^k_{(z,{\bf q}, {\bf s})}f_2(\rx,\mclW^*-\mclW_0, z, {\bf q}, {\bf s},\der_2 \mclW^*)|\le CM\sigma\quad\tx{for}\;\;k=0,1,
\end{split}
\end{equation}
for all $\rx\in\N$ where $C$ depends only on the data. Here, $F_{i}$, $f_{i}$ ($i=1, 2$) and $g$ are defined in \eqref{nf-F}, \eqref{nf-f1}, \eqref{nf-3-b5} and \eqref{defg}.

For any $\til{\mclU}=(\til \Psi, \til \phi, \til \psi) \in \mcl{K}(M_1)$, let  $\til f_2=f_2(\rx,\mclW^*-\mclW_0,\til \Psi,\nabla \til{\phi},\nabla \til{\psi},\partial_2 \mclW^*)$. It follows from  \cite{Ha-L, BDX, Bae-F}  that
 the boundary value problem
\begin{equation}
\label{2d-4-b3}
\Delta\psi=\til f_2\;\;\tx{in}\;\;\N
\end{equation}
with boundary condition \eqref{s-bc-psi}
has a unique solution $\psi\in C^{1,\alp}(\ol{\N})\cap C^{2,\alp}(\N)$ satisfying
\begin{equation}
\label{2d-4-estpsi}
\|\psi\|_{2,\alp,\N}^{(-1-\alp,\corners)}\le C_2^\sharp \|\til f_2\|_{\alp,\N}
\le C_1^{\flat}\sigma\left(M+1\right)
\end{equation}
with the constants $C_2^{\sharp}$ and $C_1^\flat$  depending only on the data and $\alp$.

Set $(\til{F}_{1},\til{F}_2,\til f_1)=({F}_{1},{F}_2,f_1)(\rx,\tilde{Q})$, and $\til g=g(\rx,\til{\varsigma},\til{\eta},\til{\bf q}, {\bf s})$ are well defined by \eqref{nf-F}, \eqref{nf-f1}, \eqref{nf-3-b5} and \eqref{defg} for
$\tilde{Q}=(\til{\varsigma},\til{\eta},\til z,\til{\bf q},{\bf s})
=(\mclW^*-\mclW_0, \til{\Psi},\nabla\til{\phi},\nabla {\psi})$ where $\psi$ is the solution for the boundary value problem \eqref{2d-4-b3} and \eqref{s-bc-psi}. Set $\til{\bf F}=(\til F_1, \til F_2)$. It follows from Lemma \ref{2d-lemma-2} and \eqref{2d-4-d1} that the linear boundary value problem \eqref{2d-4-b1}-\eqref{linearbcPsi} with ${\mfF}$, $\mff_1$, and $\mfg$ replaced by $\til{\bf F}$, $\til f_1$, and $\til g$, respectively
has a unique solutions $(\Psi,\phi)\in [C^{1,\alp}(\ol{\N})\cap C^{2,\alp}(\N)]^2$ which satisfies
\begin{equation}
\label{2d-4-d2}
\|(\Psi,\phi)\|_{2,\alp,\N}^{(-1-\alp,\corners)}
\le C_2^{\flat}(\left(M+1\right)\sigma +
\left(M_1\sigma\right)^2)
\end{equation}
with the constant $C_2^\flat$ depending only on the data and $\alpha$.
 Choose constants $M_1$ and $\hat{\sigma}$ as
\begin{equation}
\label{2d-4-d3}
M_1=2\max(1, C_1^{\flat}\left(M+1), 2C_2^{\flat}(M+1)\right),\quad
\hat{\sigma}=\frac 12
\min\left(\frac{1}{4C_2^\flat M_1}, \frac{1}{M_1^2}, \frac{\delta_0}{M_1}\right)
\end{equation}
so that \eqref{2d-4-d2} and \eqref{2d-4-estpsi} imply that $\mclU=(\Psi, \phi, \psi)$ satisfies
\begin{equation}
\label{2d-4-d4}
\begin{split}
&\|\mclU\|_{2,\alp,\N}^{(-1-\alp,\corners)}
\le M_1\sigma.
\end{split}
\end{equation}

Define an iteration mapping $\mcl{I}^{\mclW^*}$ by $\mcl{I}^{\mclW^*}(\til{\mclU})=\mclU$. Then \eqref{2d-4-d4} implies that $\mcl{I}$ maps $\mcl{K}(M_1)$ into itself  for any $\sigma\le \hat{\sigma}$.

For $i=1$, $2$, let $\til f_2^{(i)} =f_2(\mclW^*-\mclW_0,\til \Psi^{(i)},\nabla\til \phi^{(i)}, \nabla\til \psi^{(i)}, \partial_2 \mclW^*)$ and ${\bf \til F}^{(i)}$, $\til f_1^{(i)}$, and $\til g^{(i)}$ be given by \eqref{nf-F}, \eqref{nf-f1}, \eqref{nf-3-b5} and \eqref{defg} for ${Q}=(\mclW^*-\mclW_0,\til \Psi^{(i)},\nabla\til \phi^{(i)}, \nabla\psi^{(i)})$. Then the straightforward computations give
\begin{equation*}
\|\til {f}_2^{(1)}-\til {f}_2^{(2)}\|_{\alp,\N}
\le CM \sigma\|\til \mclU^{(1)}-\til \mclU^{(2)}\|_{2,\alp,\N}^{(-1-\alp,\corners)}
\end{equation*}
and
\begin{align*}
&\|\til{\bf F}^{(1)}-\til{\bf F}^{(2)}\|_{1,\alp,\N}^{(-\alp,\corners)}
+ \|\til{f}_1^{(1)}-\til{f}_1^{(2)}\|_{1,\alp,\N}^{(-\alp,\corners)} +
\|\til g^{(1)}-\til g^{(2)}\|_{1,\alp,\Gamex}^{(-\alp,\der\Gamex)}\\
\le & C\left(
(M+M_1+1)\sigma\|\til \mclU^{(1)}-\til \mclU^{(2)}\|_{1,\alp,\N}^{(-\alp,\corners)}
+\|\psi^{(1)}-\psi^{(2)}\|_{2,\alp\N}^{(-1-\alp,\corners)}\right)
\end{align*}
for a constant $C$ depending only on the data and $\alp$. It follows from  \eqref{2d-4-estpsi} and Lemma \ref{2d-lemma-2} that
\begin{align*}
\|\psi^{(1)}-\psi^{(2)}\|_{2,\alp,\N}^{(-1-\alp,\corners)}\le &
CM\sigma \|\til \mclU^{(1)}-\til \mclU^{(2)}\|_{1,\alp,\N}^{(-\alp,\corners)},
\end{align*}
and
\begin{align*}
\|(\Psi^{(1)}-\Psi^{(2)}, \phi^{(1)}-\phi^{(2)})\|_{2,\alp,\N}^{(-1-\alp,\corners)}
\le & C
(M+M_1+1)\sigma\|\til \mclU^{(1)}-\til \mclU^{(2)}\|_{1,\alp,\N}^{(-\alp,\corners)}\\
&+ C
\|\psi^{(1)}-\psi^{(2)}\|_{1,\alp,\N}^{(-\alp,\corners)} ,
\end{align*}
where the constant $C$ depends only on the data and $\alpha$.
Therefore,
\begin{equation}
\label{2d-3-d1}
\|\mclU^{(1)}-\mclU^{(2)}\|_{2,\alp,\N}^{(-1-\alp,\corners)}
\le C_3^\flat(M+M_1+1)\sigma\|\til \mclU^{(1)}-\til \mclU^{(2)}\|_{2,\alp,\N}^{(-1-\alp,\corners)}
\end{equation}
where the constant $C_3^\flat$ depends only on the data and $\alp$. Finally, choose $\sigma_5$ to be
\begin{equation}
\label{2d-3-d2}
\sigma_5=\min(\hat{\sigma}, \frac{4}{5C_3^\flat(M+M_1+1)})
\end{equation}
where $\hat\sigma$ is defined in \eqref{2d-4-d3}.
Thus if $\sigma<\sigma_5$, then the mapping $\mcl{I}^{\mclW^*}$ is a contraction mapping  so that  $\mcl{I}^{\mclW^*}$ has a unique fixed point in $\mcl{K}(M_1)$. This also gives the unique existence of a solution to \eqref{2d-4-a2}--\eqref{2d-4-a5} with \eqref{2d-3-c2}.
\end{proof}

\section{Proof of Theorem \ref{nf-theorem3}}\label{section-pf-theorem2}
In this section, we give the proof of Theorem \ref{nf-theorem3}. The stream function formulation plays an important role.

\begin{proof}
[Proof of Theorem \ref{nf-theorem3}]
To prove Theorem \ref{nf-theorem3}, it suffices to show that there exists unique $(\mclU, \mclW)\in [C^1(\ol{\N})\cap C^2(\N)]^3\times [C(\ol{\N})\cap C^1({\N})]^2$ solving  \eqref{nf-3-c1}, \eqref{nf-3-b2} and \eqref{eqSK} in $\N$ with boundary conditions \eqref{s-bc-psi}, \eqref{bcSK}, \eqref{3-d3}, \eqref{2d-3-b2} and \eqref{2d-3-bcPsi}. Since unique solvability of \eqref{2d-4-a2}--\eqref{2d-4-a5} with boundary conditions \eqref{s-bc-psi}, \eqref{bcSK},  \eqref{2d-3-b2}, \eqref{2d-3-bcPsi}, and \eqref{2d-4-a3} associated with $\mclW^*$ is guaranteed by Proposition \ref{proposition-1}, we prove Theorem \ref{nf-theorem3} by an iteration scheme for $\mclW$.

Let $\sigma_5$ in Proposition \ref{proposition-1} be given by \eqref{2d-3-d2}.  Then
Proposition \ref{proposition-1} implies that if $\sigma\le \sigma_5$, then for any fixed $\mclW^* \in \mcl{P}(M)$, the nonlinear boundary value problem \eqref{2d-4-a2}--\eqref{2d-4-a5} with boundary conditions \eqref{s-bc-psi}, \eqref{bcSK},  \eqref{2d-3-b2}, \eqref{2d-3-bcPsi}, and \eqref{2d-4-a3} has a unique solution $\mclU$ satisfying the estimate \eqref{2d-3-c2}.


The straightforward computations show that ${\bf V}$ defined in \eqref{v} satisfies
\begin{equation}
\|{\bf V}-(J_0, 0)\|_{1, \alpha, \N}^{(-\alpha, \Gamw)}\leq C_4^\flat (M +M_1)\sigma
\end{equation}
where the constant $C_4^\flat$ depends only on the data and $\alpha$. Let $C_0$ be $C^*$ in \eqref{nf-4-c3} associated with $K_0=2J_0$ and $\nu^*=J_0/2$. Choose
\begin{equation}
\label{2d-4-e3}
M=2 C_0,\quad \sigma_3=\min\{\sigma_5, \frac{J_0}{2C_4^\flat(M+M_1)}\}.
\end{equation}
Then Lemma \ref{nf-lemma-4-1} shows that \eqref{2d-3-b4} has a unique solution $\mclW$ satisfying  \eqref{nf-4-c3}.
Thus the mapping \eqref{2d-3-b5} maps $\mcl{P}(M)$ into itself.

Note that $\mcl{P}(M)$ is convex and a compact subset of $[C^{1,\frac{\alp}{2}}(\ol{\N})]^2$. Take a sequence $\{\mclW^{*,(k)}\}_{k=1}^{\infty}\subset\mcl{P}(M)$ convergent in $C^{1,\frac{\alp}{2}}(\ol{\N})$  to $\mclW^{*, (\infty)}\in \mcl{P}(M)$. Set $ \mclW^{(k)}:=\mcl{J}(\mclW^{*, (k)})$ for each $k\in\mathbb{N}$ and $ \mclW^{(\infty)}:=\mcl{J}(\mclW^{*,(\infty)})$.

For $k\in \mathbb{N}$, let ${\mcl{U}}^{(k)}:=(\Psi^{(k)},\phi^{(k)},\psi^{(k)})\in \mathcal{K}(M_1)$ be the unique solution of \eqref{2d-4-a2}--\eqref{2d-4-a5} with boundary conditions \eqref{s-bc-psi}, \eqref{bcSK},  \eqref{2d-3-b2}, \eqref{2d-3-bcPsi}, and \eqref{2d-4-a3} associated with $\mclW^*=\mclW^{*,(k)}$. Therefore, $\{\mcl{U}^{(k)}\}$ is sequentially compact in $C_{(-1-\alpha/2, \Gamw)}^{2, \alpha/2}(\N)$. Define ${\bf V}^{(k)}$ by
\begin{equation*}
{\bf V}^{(k)}=H(S^{*,(k)},\msK^{*,(k)}+\Phi_0+\Psi^{(k)}-\frac 12|\nabla\vphi_0+\nabla\phi^{(k)}+\nabla^{\perp}\psi^{(k)}|^2)
(\nabla\vphi_0+\nabla\phi^{(k)}+\nabla^{\perp}\psi^{(k)}).
\end{equation*}
 Note that the limit of each convergent subsequence of $\{\mclU^{(k)}\}$ in $C_{(-1-\alpha/2, \Gamw)}^{2, \alpha/2}(\N)$ solves \eqref{2d-4-a2}--\eqref{2d-4-a5} with boundary conditions \eqref{s-bc-psi}, \eqref{bcSK},  \eqref{2d-3-b2}, \eqref{2d-3-bcPsi}, and \eqref{2d-4-a3} associated with $\mclW^*=\mclW^{*,(\infty)}$. Therefore, $\{\mclU^{(k)}\}_{k=1}^\infty$ is convergent in $C_{(-1-\alpha/2, \Gamw)}^{2, \alpha/2}(\N)$. Denote its limit by  $\mclU^{(\infty)}$. Thus ${\bf V}^{(k)}$ converges to
\begin{equation*}
{\bf V}^{(\infty)}=H(S^{(\infty)},\msK^{(\infty)}+\Phi_0+\Psi^{(\infty)}-\frac 12|\nabla\vphi_0+\nabla\phi^{(\infty)}+\nabla^{\perp}\psi^{(\infty)}|^2)
(\nabla\vphi_0+\nabla\phi^{(\infty)}+\nabla^{\perp}\psi^{(\infty)})
\end{equation*}
in $C^{\alpha/2}(\bar\N)$.
Hence it follows from \eqref{2d-3-c6} and \eqref{gradL} that $\mclW^{(k)}$ converges to $\mclW^{(\infty)}$ in $C^{1, \alpha/2}(\bar\N)$  which is the solution of \eqref{2d-3-b4} associated with ${\bf V}^{(\infty)}$. This means that  $\mcl{J}(\mclW^{*,(k)})$ converges to $\mcl{J}(\mclW^{*,(\infty)})$ in $C^{1, \alpha/2}(\bar \N)$.
Thus $\mcl{J}$ is a continuous map in $C^{1, \alpha/2}(\bar\N)$ .
Applying the Schauder fixed point theorem (\cite[Theorem 11.1]{GilbargTrudinger}) yields that $\mcl{J}$ has a fixed point. This proves (a) of Theorem \ref{nf-theorem3}.


Let $(\mclU^{(i)},\mclW^{(i)})$ ($i=1, 2$)  be two solutions of \eqref{nf-3-c1}, \eqref{nf-3-b2} and \eqref{eqSK} in $\N$ with boundary conditions  \eqref{s-bc-psi}, \eqref{bcSK}, \eqref{3-d3}, \eqref{2d-3-b2} and \eqref{2d-3-bcPsi} satisfying \eqref{theorem3-est2}.  Assume that $\om_4(\Sen,\msB_{en})$  defined by \eqref{theorem1-as1} is finite for some $\mu>2$. Choose $\mu_1\in (2, \min\{\mu, \frac{1}{1-\alpha}\})$.
Set $\beta=\frac 12\min(\alp, 1-\frac{2}{\mu_1})$. We  show that there exists a constant $C$ depending only on the data, $\alp$ and $\mu$ to satisfy
\begin{equation}
\label{2d-4-f1}
\begin{split}
&\|{\mcl{U}}^{(1)} -\mcl{U}^{(2)} \|_{1,\beta,\N} \le C\kappa \|{\mcl{U}}^{(1)} -\mcl{U}^{(2)} \|_{1,\beta,\N},
\end{split}
\end{equation}
where
\begin{equation*}
\kappa:= \sigma +\om_4(\Sen, \msB_{en}, \Phi_{bd})
\end{equation*}
with  $\sigma = \om_1(b)+\om_2(\Sen, \msB_{en})+\om_3(\Phi_{bd},\pex)$.
Once \eqref{2d-4-f1} is verified, one can choose a constant $\sigma_4\in(0,\sigma_3]$ so that if $\kappa\le \sigma_4$, then \eqref{2d-4-f1} implies $\mclU^{(1)}=\mclU^{(2)}$ in $\N$. Hence it follows from Lemma \ref{nf-lemma-4-1}  that
 $\mclW^{(1)}=\mclW^{(2)}$. This proves the uniqueness of solutions. Thus the rest of the proof is devoted to
 verify \eqref{2d-4-f1}.

For each $j=1$ and $2$, let ${\bf F}^{(j)}, f_1^{(j)}, f_2^{(j)}$ and $g^{(j)}$ be given by \eqref{nf-F}, \eqref{nf-f1}, \eqref{nf-3-b5} and \eqref{defg} associated with ${Q}=(\mclW^{(j)}-\mclW_0, \Psi^{(j)}, \nabla\phi^{(j)}, \nabla\psi^{(j)})$ and $({\xi},{\tau})=\der_2 \mclW^{(j)}$. It follows from \eqref{2d-4-c6} and \eqref{2d-4-c7} that one has
\begin{equation}\label{4-star1}
\begin{aligned}
&\|(\phi^{(1)}-\phi^{(2)}, \Psi^{(1)}-\Psi^{(2)})\|_{1,\beta,\N}\\
\le &
C ( \|{\bf F}^{(1)}-{\bf F}^{(2)}\|_{\beta,\N}
+\|{f_1}^{(1)}-{f_1}^{(2)}\|_{0,\N}
+\|{g}^{(1)}-{g}^{(2)}\|_{\beta,\N}).
\end{aligned}
\end{equation}
From now on, the constant $C$ depends only on the data, $\beta$, $\mu$, and $\N$, which may vary from line to line. It follows from \cite[Theorem 3.13]{Ha-L} and the method of reflection that the difference of $\psi^{(1)}$ and $\psi^{(2)}$ satisfies
\begin{equation}\label{w2pest}
\|\psi^{(1)}-\psi^{(2)}\|_{1, \beta, \N}\leq C \|f_2^{(1)}-f_2^{(2)}\|_{L^{\mu_1}(\N)}.
\end{equation}
The straightforward computations give
\begin{equation}
\label{2d-4-f2}
\begin{split}
&\|{\bf F}^{(1)}-{\bf F}^{(2)}\|_{\beta,\N}
+\|{f_1}^{(1)}-{f_1}^{(2)}\|_{0,\N}
+\|{g}^{(1)}-{g}^{(2)}\|_{\beta,\N}\le  C\Big(\|\psi^{(1)}-\psi^{(2)}\|_{1,\beta,\N}\\
& + (M+M_1+1)\sigma \|(\phi^{(1)}-\phi^{(2)}, \Psi^{(1)}-\Psi^{(2)})\|_{1,\beta,\N} +
\|\mclW^{(1)}-\mclW^{(2)}\|_{\beta,\N}\Big)
\end{split}
\end{equation}
and
\begin{equation}
\label{2d-4-f3}
\begin{split}
&\|f_2^{(1)}-f_2^{(2)}\|_{L^{\mu_1}(\N)}\\
\le &
C\Bigl(M\sigma \|\mcl{U}^{(1)}-\mcl{U}^{(2)}\|_{1,\beta,\N}
+\|\mclW^{(1)}-\mclW^{(2)}\|_{0,\N}+
\|\der_2\mclW^{(1)}-\der_2\mclW^{(2)}\|_{L^{\mu_1}(\N)} \Bigr).
\end{split}
\end{equation}

Now we need to estimate $\|\mclW^{(1)}-\mclW^{(2)}\|_{\beta,\N}$ and $\|\der_2\mclW^{(1)}-\der_2\mclW^{(2)}\|_{L^{\mu_1}(\N)}$.
For each $j=1$ and $2$, $\mclW^{(j)}=\mclWen\circ\mathscr{L}^{(j)}$ where $\mathscr{L}^{(j)}$ is given by \eqref{L} associated with the vector field
\begin{equation}
\label{2d-4-f4}
{\bf V}^{(j)}=H(S^{(j)}, \msK^{(j)}+\Phi_0+\Psi^{(j)}-\frac 12|\nabla\vphi_0+\nabla\phi^{(j)}+\nabla^{\perp}\psi^{(j)}|^2)
(\nabla\vphi_0+\nabla\phi^{(j)}+\nabla^{\perp}\psi^{(j)}).
\end{equation}
Note that
\begin{equation}\label{estwb}
\begin{aligned}
\|\mclW^{(1)}-\mclW^{(2)}\|_{\beta, \N} = &\|\mclWen\circ\mathscr{L}^{(1)}(\rx) - \mclWen\circ\mathscr{L}^{(2)}(\rx)\|_{\beta, \N}\\
= &\|\int_0^1 \mclWen'(\theta \mathscr{L}^{(1)} +(1-\theta) \mathscr{L}^{(1)}) d\theta (\mathscr{L}^{(1)}-\mathscr{L}^{(2)})\|_{\beta, \N}\\
\leq & \|\mclWen'\|_{\beta, \N} \|\mathscr{L}^{(1)}-\mathscr{L}^{(2)}\|_{\beta, \N}\\
\leq & C\sigma  \|\mathscr{L}^{(1)}-\mathscr{L}^{(2)}\|_{\beta, \N}.
\end{aligned}
\end{equation}
It follows from \eqref{2d-3-c4} and \eqref{L} that
\begin{equation*}
w^{(1)}(0,\mathscr{L}^{(1)}(\rx))-w^{(2)}(0,\mathscr{L}^{(2)}(\rx))
=w^{(1)}(\rx)-w^{(2)}(\rx)
\end{equation*}
for all $\rx\in\N$, where  $w^{(j)}$ is defined by \eqref{w} corresponding to ${\bf V}={\bf V}^{(j)}$ and satisfies
$\nabla^{\perp}w^{(j)}={\bf V}^{(j)}$.
 Therefore,
\begin{equation}
\label{2d-4-f6}
(\mathscr{L}^{(1)}-\mathscr{L}^{(2)})(\rx)
=\frac{(w^{(1)}-w^{(2)})(\rx)-(w^{(1)}-w^{(2)})(0,\mathscr{L}^{(2)}(\rx))}
{\int_0^1 V_1^{(1)}(0,t\mathscr{L}^{(1)}(\rx)+(1-t)\mathscr{L}^{(2)}(\rx))\;dt}.
\end{equation}
Thus we have
\begin{equation}\label{estLdiff}
\|(\mathscr{L}^{(1)}-\mathscr{L}^{(2)})(\rx)\|_{\beta, \N}\leq C(\|\mcl{U}^{(1)}-\mcl{U}^{(2)}\|_{1,\beta,\N}+ \|\mclW^{(1)}-\mclW^{(2)}\|_{\beta,\N}).
\end{equation}
Combining \eqref{estwb} and \eqref{estLdiff} together and choosing $\sigma$ suitably small yield
\begin{equation}
\label{2d-4-f8}
\|\mclW^{(1)}-\mclW^{(2)}\|_{\beta,\N}
\le C\sigma\|\mcl{U}^{(1)}-\mcl{U}^{(2)}\|_{1,\beta,\N}.
\end{equation}
Furthermore, the direct computation gives
$$\der_2\mclW^{(j)}(\rx)=
\mclWen'(\mathscr{L}^{(j)}(\rx))\der_2\mathscr{L}^{(j)}(\rx)\,\,\text{for}\,\,  j=1\,\,\text{and}\,\,  2$$
and
\begin{equation}
\label{2d-4-f9}
\begin{split}
\|\der_2(\mclW^{(1)}-\mclW^{(2)})\|_{L^{\mu_1}(\N)}\le & \|\mclWen'(\mathscr{L}^{(1)})-\mclWen'(\mathscr{L}^{(2)})\|_{L^{\mu_1}(\N)}
\|\mathscr{L}^{(1)}\|_{0,\N}\\
&+
\|\mclWen'\|_{0,\N}
\|\der_2(\mathscr{L}^{(1)}-\mathscr{L}^{(2)})\|_{L^{\mu_1}(\N)}
\end{split}
\end{equation}
Note that \eqref{gradL} gives
\begin{equation}\label{estkey1}
\der_2\mathscr{L}^{(1)}-\der_2\mathscr{L}^{(2)}=\frac{V_1^{(1)}(\rx)}{V_1^{(1)}(0, \mathscr{L}^{(1)}(\rx))}- \frac{V_1^{(2)}(\rx)}{V_1^{(2)}(0, \mathscr{L}^{(2)}(\rx))}.
\end{equation}
Since $\phi^{(i)} \in C_{2, \alpha, \N}^{(-1-\alpha, \Gamw)}$ for $i=1$, 2, one has $\phi^{(i)} \in C_{2, \N}^{(-1-\alpha, \Gamw)}$ so that
\begin{equation}
|\partial \nabla \phi(\rx)|\leq \frac{\|\phi\|_{2, \alpha, \N}^{(-1-\alpha, \Gamw)}}{|x_2(1-x_2)|^{1-\alpha}}.
\end{equation}
 Thus for $\mu_1\in (0, \frac{1}{1-\alpha})$, we have
\begin{equation}\label{estkey2}
\begin{aligned}
& \|\nabla \phi^{(1)}(0, \mathscr{L}^{(1)}) -\nabla \phi^{(2)}(0, \mathscr{L}^{(2)})\|_{L^{\mu_1}(\N)} \\
\le & C \left\|\int_0^1 \der_2\nabla \phi^{(1)}(0, t\mathscr{L}^{(1)}+(1-t)\mathscr{L}^{(2)} )dt\right\|_{L^{\mu_1}(\N)} \|\mathscr{L}^{(1)}-\mathscr{L}^{(2)})\|_{L^\infty(\N)}\\
& +C \|(\nabla \phi^{(1)}-\nabla \phi^{(2)})(0, \mathscr{L}^{(2)})\|_{L^{\mu_1}(\N)}\\
\leq & C \|\phi\|_{2, \alpha, \N}^{(-1-\alpha, \Gamw)}(\|\mcl{U}^{(1)}-\mcl{U}^{(2)}\|_{1,\beta,\N}+ \|\mclW^{(1)}-\mclW^{(2)}\|_{\beta,\N}) + C\|\mclU^{(1)}-\mclU^{(2)}\|_{1,\N}.
 \end{aligned}
 \end{equation}
Similarly, we have
\begin{equation}\label{estkey3}
\begin{aligned}
& \|\nabla \mclU^{(1)}(0, \mathscr{L}^{(1)}) -\nabla \mclU^{(2)}(0, \mathscr{L}^{(2)})\|_{L^{\mu_1}(\N)} \\
\leq & C \|\mclU\|_{2, \alpha, \N}^{(-1-\alpha, \Gamw)}(\|\mcl{U}^{(1)}-\mcl{U}^{(2)}\|_{1,\beta,\N}+ \|\mclW^{(1)}-\mclW^{(2)}\|_{\beta,\N}) + C\|\mclU^{(1)}-\mclU^{(2)}\|_{1,\N}.
 \end{aligned}
\end{equation}
Furthermore, we have
\begin{equation}\label{estkey4}
\begin{aligned}
& \| \mclW^{(1)}(0, \mathscr{L}^{(1)}) - \mclW^{(2)}(0, \mathscr{L}^{(2)})\|_{L^{\mu_1}(\N)} =  \| \mclWen(\mathscr{L}^{(1)}) - \mclWen( \mathscr{L}^{(2)})\|_{L^{\mu_1}(\N)}  \\
\leq & C \sigma(\|\mcl{U}^{(1)}-\mcl{U}^{(2)}\|_{1,\beta,\N}+ \|\mclW^{(1)}-\mclW^{(2)}\|_{\beta,\N}).
 \end{aligned}
\end{equation}
Combining the estimate \eqref{2d-4-f9}-\eqref{estkey4} yields
\begin{equation}\label{2d-4-W}
\|\der_2(\mclW^{(1)}-\mclW^{(2)})\|_{L^{\mu_1}(\N)}\le C\kappa (
\|\mcl{U}^{(1)}-\mcl{U}^{(2)}\|_{1,\beta,\N} + \|\mclW^{(1)}-\mclW^{(2)}\|_{\beta,\N})).
\end{equation}
Substituting \eqref{2d-4-f8}-\eqref{2d-4-W} into \eqref{2d-4-f2} and \eqref{2d-4-f3} gives
\begin{equation}
\label{2d-4-g1}
\|{\bf F}^{(1)}-{\bf F}^{(2)}\|_{\beta,\N}
+\|{f_1}^{(1)}-{f_1}^{(2)}\|_{0,\N}
+\|{g}^{(1)}-{g}^{(2)}\|_{\beta,\N}
\le C\kappa\|\mcl{U}^{(1)}-\mcl{U}^{(2)}\|_{1,\beta,\N},
\end{equation}
and
\begin{equation}\label{2d-4-g11}
\|f_2^{(1)}-f_2^{(2)}\|_{L^{\mu_1}(\N)}
\le C\kappa\|\mcl{U}^{(1)}-\mcl{U}^{(2)}\|_{1,\beta,\N}.
\end{equation}
It follows from \eqref{4-star1} and  \eqref{2d-4-g1} that we have
\begin{equation}\label{estdiff1}
\|(\phi^{(1)}-\phi^{(2)}, \Psi^{(1)}-\Psi^{(2)})\|_{1,\beta,\N}\le
C_5^{\flat}\kappa \|\mcl{U}^{(1)}-\mcl{U}^{(2)}\|_{1,\beta,\N}
\end{equation}
for a constant $C_5^{\flat}$ depending on the data, $\alp$ and $\mu$. Similarly,  \eqref{w2pest} and  \eqref{2d-4-g11} yield
\begin{equation}\label{estdiff2}
\|\psi^{(1)}-\psi^{(2)}\|_{1,\beta,\N}\le C_6^{\flat}\kappa\|\mcl{U}^{(1)}-\mcl{U}^{(2)}\|_{1,\beta,\N}
\end{equation}
where the constant $C_6^{\flat}$ depends on the data and $\mu$.  Choose $\sigma_4=\min(\sigma_5, \frac{4}{5(C_5^{\flat}+C_6^{\flat})})$. Then \eqref{2d-4-f1} is a direct consequence of \eqref{estdiff1} and \eqref{estdiff2}. The proof is complete.
\end{proof}

\bigskip

{\bf Acknowledgement.}
The research of Myoungjean Bae and Ben Duan was supported in part by Priority Research Centers Program through the National Research Foundation of Korea(NRF) (NRF Project no. 2012047640). The research of Myoungjean Bae was also supported by the Basic Science Research Program(NRF-2012R1A1A1001919) and TJ Park Science Fellowship of POSCO TJ Park Foundation.
The research of Xie was supported in part by  NSFC 11241001 and 11201297, Shanghai Chenguang program and Shanghai Pujiang  program 12PJ1405200.
Part of the work was done when Xie was visiting POSTECH and Bae and Duan visited  Shanghai Jiao Tong University. They thank these institutions for their hospitality and support during these visit.

\end{document}